\documentclass[a4paper,11pt]{amsart}
\usepackage{a4wide}
\usepackage{enumerate}
\usepackage[utf8]{inputenc}
\usepackage{url}
\usepackage{booktabs}          
\usepackage{float}             
\usepackage{lscape}            
\usepackage{changepage}        
\usepackage{comment}
\usepackage{hyperref,srcltx}
\usepackage{array}

\numberwithin{equation}{section}

\newcommand{\field}[1]{\mathbb{{#1}}}

\newcommand{\R}{\field{R}}
\newcommand{\C}{\field{C}}
\newcommand{\N}{\field{{N}}}

\newcommand{\Q}{\field{{Q}}}
\newcommand{\K}{\field{{K}}}
\newcommand{\intpart}[1]{\left\lfloor#1\right\rfloor}
\newcommand{\disc}{\Delta}

\newcommand{\mltc}{\multicolumn}
\newcommand{\RR}{{\mathcal{R}}}
\newcommand{\dd}{\,\mathrm{d}}
\newcommand{\PP}{\mathfrak{p}}
\newcommand{\Norm}{\textrm{\upshape N}}
\newcommand{\ExpInt}{\mathop{\text{\upshape E}_1}}
\newcommand{\zZ}{Z}
\newcommand{\eul}{\varphi}

\DeclareMathOperator{\Ree}{Re}
\DeclareMathOperator{\res}{res}
\DeclareMathOperator{\Gal}{Gal}

\DeclareMathOperator{\id}{id}
\DeclareMathOperator{\bas}{bas}
\DeclareMathOperator{\imp}{imp}

\newtheorem{theorem}{Theorem}[section]
\newtheorem{lemma}[theorem]{Lemma}
\newtheorem{corollary}[theorem]{Corollary}

\theoremstyle{remark}
\newtheorem{remark}[theorem]{Remark}
\newtheorem*{remark*}{Remark}
\newtheorem*{acknowledgements}{Acknowledgements}
\newtheorem*{notation}{Notation}

\allowdisplaybreaks[4]

\begin{document}
\title
{Explicit smoothed prime ideals theorems under GRH}

\author[L.~Greni\'{e}]{Lo\"{i}c Greni\'{e}}
\address[L.~Greni\'{e}]{Dipartimento di Ingegneria gestionale, dell'informazione e della produzione\\
         Universit\`{a} di Bergamo\\
         viale Marconi 5\\
         I-24044 Dal\-mi\-ne\\
         Italy}
\email{loic.grenie@gmail.com}

\author[G.~Molteni]{Giuseppe Molteni}
\address[G.~Molteni]{Dipartimento di Matematica\\
         Universit\`{a} di Milano\\
         via Saldini 50\\
         I-20133 Milano\\
         Italy}
\email{giuseppe.molteni1@unimi.it}


\date{\today.
}

\begin{abstract}
Let $\psi_{\K}$ be the Chebyshev function of a number field $\K$. Let
$\psi^{(1)}_\K(x)\!:=\!\int_{0}^{x}\psi_\K(t)\dd t$ and
$\psi^{(2)}_\K(x):=2\int_{0}^{x}\psi^{(1)}_\K(t)\dd t$. We prove under GRH (Generalized Riemann
Hypothesis) explicit inequalities for the differences $|\psi^{(1)}_\K(x) - \tfrac{x^2}{2}|$ and
$|\psi^{(2)}_\K(x) - \tfrac{x^3}{3}|$. We deduce an efficient algorithm for the computation of the
residue of the Dedekind zeta function and a bound on small-norm prime ideals.
\end{abstract}

\maketitle

\begin{center}
To appear in Math. Comp. 2015.
\end{center}

\section{Introduction}\label{sec:B1}
For a number field $\K$ we denote
\begin{itemize}
\item[] $n_\K$ its dimension,
\item[] $\disc_\K$ the absolute value of its discriminant,
\item[] $r_1$ the number of its real places,
\item[] $r_2$ the number of its imaginary places,
\item[] $d_\K := r_1+r_2-1$.
\end{itemize}
Moreover, throughout this paper $\PP$ denotes a maximal ideal of the integer ring $\mathcal{O}_\K$ and
$\Norm\PP$ its absolute norm. The von Mangoldt function $\Lambda_\K$ is defined on the set of ideals of
$\mathcal{O}_\K$ as $\Lambda_\K(\mathfrak{I}) = \log\Norm\PP$ if $\mathfrak{I}=\PP^m$ for some $\PP$ and
$m\geq 1$, and is zero otherwise. Moreover, the Chebyshev function $\psi_\K$ and the arithmetical function
$\tilde{\Lambda}_\K$ are defined via the equalities
\[
\psi_\K(x)
:=  \sum_{\substack{\mathfrak{I}\subset \mathcal{O}_\K\\ \Norm\mathfrak{I} \leq x}} \Lambda_\K(\mathfrak{I})
 =: \sum_{n \leq x} \tilde{\Lambda}_\K(n).
\]
In 1979, Oesterl\'{e} announced~\cite{Oesterle} a general result implying under the Generalized Riemann
Hypothesis that
\begin{equation}\label{eq:B0}
|\psi_\K(x) - x| \leq \sqrt{x}
                              \Big[
                                   \Big(\frac{\log   x}{ \pi}+2\Big)\log\disc_\K
                                  +\Big(\frac{\log^2 x}{2\pi}+2\Big)n_\K
                              \Big]
\qquad
\forall\,x\geq 1,
\end{equation}
but its proof has never appeared. The stronger bound with $\log x$ substituted by $\tfrac{1}{2}\log x$
has been proved by the authors~\cite{GrenieMolteni3} for $x\geq 100$.

The function $\psi_\K(x)$ is the first member of a sequence of similar sums $\psi^{(m)}_\K(x)$ which are
defined for every $m\in\N$ as
\[
\psi_\K^{(0)}(x):=\psi_\K(x)
\qquad
\psi^{(m)}_\K(x):= m \int_0^x \psi^{(m-1)}_\K(u)\dd u
                 = \sum_{n\leq x} \tilde{\Lambda}_\K(n)(x-n)^m
\]
and are smoothed versions of $\psi_\K(x)$. They could be studied using~\eqref{eq:B0} via a partial
summation formula, but a direct attack via the integral identities
\begin{equation}\label{eq:D1}
\psi^{(m)}_\K(x)
=  -\frac{m!}{2\pi i}\int_{2-i\infty}^{2+i\infty} \frac{\zeta'_\K}{\zeta_\K}(s) \frac{x^{s+m}}{s(s+1)\cdots(s+m)}\dd s
\qquad
\forall x\geq 1,
\quad
\forall m\geq 0
\end{equation}
(see Section~\ref{sec:B3}) produces better results, as a consequence of the better decay that the kernel
in the integral has for $m\geq 1$ with respect to the case $m=0$. In fact, the absolute integrability of
the kernel allows us to apply the Cauchy integral formula to quickly obtain that
\begin{subequations}\label{eq:E1}
\begin{align}
\psi^{(1)}_\K(x)
&=  \frac{x^2}{2}
 - \sum_{\rho} \frac{x^{\rho+1}}{\rho(\rho+1)}
 - x r_\K
 +   r'_\K
 + R^{(1)}_{r_1,r_2}(x),                                \label{eq:E1a}\\
\psi^{(2)}_\K(x)
&=  \frac{x^3}{3}
 - \sum_{\rho} \frac{2x^{\rho+2}}{\rho(\rho+1)(\rho+2)}
 - x^2 r_\K
 + 2x  r'_\K
 -     r''_\K
 + R^{(2)}_{r_1,r_2}(x),                                \label{eq:E1b}
\end{align}
\end{subequations}
and analogous formulas for every $m\geq 3$, where $\rho$ runs on the set of nontrivial zeros for
$\zeta_\K$, the constants $r_\K$, $r'_\K$ and $r''_\K$ are defined in~\eqref{eq:B18} below and the
functions $R^{(m)}_{r_1,r_2}(x)$ in Lemma~\ref{lem:B2}. These representations show that the main term for
the difference $\psi^{m}_\K(x) - \frac{x^m}{m}$ comes from the sum on nontrivial zeros.\\
Assuming the Generalized Riemann Hypothesis we have the strongest horizontal localization on zeros but we
lack any sharp vertical information. Thus we are in some sense forced to estimate the sum with $x^{m+1/2}
\sum_\rho |\rho(\rho+1)\cdots(\rho+m)|^{-1}$, and the problem here is essentially producing good bounds
for this sum. To estimate this type of sums, we use the following method. Let $\zZ$ be the set of
imaginary parts of the nontrivial zeros of $\zeta_\K$, counted with their multiplicities, and let
\begin{align*}
f(s,\gamma) &:= \Ree\Big(\frac{2}{s-(\frac{1}{2}+i\gamma)}\Big),  \\
f_\K(s)     &:= \sum_{\gamma\in\zZ}f(s,\gamma).
\end{align*}
The sum converges to the real part of a meromorphic function with poles at the zeros of $\zeta_\K$. Let
$g$ be a non-negative function. Suppose we have a real measure $\mu$ supported on a subset $D\subseteq\C$
such that
\begin{equation}\label{eq:BB23}
g(\gamma)\leq\int_{D}f(s,\gamma)\dd\mu(s),
\qquad\forall \gamma\in\R,
\end{equation}
then under moderate conditions on $D$ and $\mu$ we have
\begin{align}
\sum_{\gamma\in\zZ}g(\gamma)
    &\leq \sum_{\gamma\in\zZ}\int_{D}f(s,\gamma)\dd\mu(s)  
     =    \int_{D}\sum_{\gamma\in\zZ}f(s,\gamma)\dd\mu(s)  
     =    \int_{D}f_\K(s)\dd\mu(s).                        \label{eq:BB24}
\end{align}
To ensure the validity of the estimate it is sufficient to have $D$ on the right of the line
$\Ree(s)=\tfrac{1}{2}+\varepsilon$ for some $\varepsilon>0$ and $\mu$ of bounded variation. The interest
of the method comes from the fact that, using the functional equation of $\zeta_\K$, one can produce a
formula for $f_\K$ independent of the zeros (see~\eqref{eq:B17}).

The aforementioned idea works very well for certain $g$ corresponding to $m=1$ and $2$ above, allowing us
to prove the explicit formulas for $\psi^{(1)}_\K(x)$ and $\psi^{(2)}_\K(x)$ given in
Theorem~\ref{th:B2}. Other applications of this idea can be found in~\cite{GrenieMolteni3}
and~\cite{GrenieMolteni4}.
\begin{theorem}\label{th:B2}
(GRH) For every $x\geq 3$, when $\K\neq \Q$ we have
\begin{align*}
\Big|\psi^{(1)}_\K(x) - \frac{x^2}{2}\Big|
\leq\,&x^{3/2}(0.5375\log\disc_\K  - 1.0355n_\K + 5.3879)
       + (n_\K-1)x\log x                                      \\
      &+ x (1.0155\log \disc_\K - 2.1041 n_\K + 8.3419)
       + \log\disc_\K - 1.415 n_\K + 4,                       \\
\Big|\psi^{(2)}_\K(x) - \frac{x^3}{3}\Big|
\leq\,&  x^{5/2}(0.3526\log\disc_\K - 0.8212 n_\K + 4.4992)
       + (n_\K-1)x^2(\log x - \tfrac{1}{2})                   \\
      &+ x^2 (1.0155\log \disc_\K - 2.1041 n_\K + 8.3419)
       + 2x  (\log \disc_\K - 1.415 n_\K + 4)                 \\
      &+     \log \disc_\K - 0.9151 n_\K + 2,
\intertext{while for $\Q$ the bounds become}
\Big|\psi^{(1)}_\Q(x) - \frac{x^2}{2}\Big|
\leq\,&  0.0462 x^{3/2}
       + 1.838  x,                                            \\
\Big|\psi^{(2)}_\Q(x) - \frac{x^3}{3}\Big|
\leq\,&  0.0015 x^{5/2}
       + 1.838  x^2.
\end{align*}
\end{theorem}
\noindent
The method can be easily adapted to every $m\geq 3$, but depends on several parameters that we have to
set in a proper way to get an interesting result, and whose dependence on $m$ is not clear. As a
consequence it is not evident that the bounds for each $m\geq 3$ will be as good as the cases $m=1$ and
$2$, despite the fact that our computations for $m=3$ and $4$ show that it should be possible. Moreover,
the applications we will show in the next section essentially do not benefit from any such extension, the
cases $m=1$ and $2$ giving already the best conclusions (see Remark~\ref{rem:B2} below). Thus we have
decided not to include the cases $m=3$ and $4$ in the paper.

\begin{remark}\label{rem:Bintro}
Integrating \eqref{eq:BB23} for $\gamma\in\R$ we find that, if $D$ is in the $\Ree s>\tfrac{1}{2}$ half
of the plane, then
\[
\mu(D) \geq \frac{1}{2\pi}\int_\R g(\gamma)\dd\gamma.
\]
The measure $\mu(D)$ will end up as the main coefficient of $\log\disc_\K$ in our inequalities. This means
that the coefficient of $\log\disc_\K$ that we can obtain with this method is necessarily greater than
$\frac{1}{2\pi}\int_\R g(\gamma)\dd\gamma$.\\
Finally, we notice that our method is not limited to upper-bounds, since if we change $\leq$ to $\geq$ in
Inequality~\eqref{eq:BB23}, then Inequality~\eqref{eq:BB24} gives a lower bound. For an application see
Remark~\ref{rem:B3} below.
\end{remark}

A file containing the PARI/GP~\cite{PARI2} code we have used for a set of computations is available at the
following address:\\
\url{http://users.mat.unimi.it/users/molteni/research/psi_m_GRH/psi_m_GRH_data.gp}.

\begin{notation}
$\intpart{x}$ denotes the integral part of $x$; $\gamma$ denotes the imaginary part of the nontrivial
zeros, but in some places it will denote also the Euler-Mascheroni constant, the actual meaning being
clear from the context.
\end{notation}

\begin{acknowledgements}
Special thanks go to Alberto Perelli and Karim Belabas for their valuable remarks and comments, and to
Michael Rubinstein, who provided the authors the zeros for many Dirichlet $L$-functions. We are very
grateful to the referee for her/his suggestions which greatly improved the presentation.
\end{acknowledgements}

\section{Applications}\label{sec:C1}
\subsection*{Small prime ideals}
The bound in~\eqref{eq:B0} can be used to prove that $\psi_\K(x)>0$ when $x\geq
4(\log\disc_\K\log^2\log\disc_\K + 5n_\K + 10)^2$.
%
%
This fact, without explicit constants, was already mentioned by Lagarias and
Odlyzko~\cite{LagariasOdlyzko} who also gave an argument to remove the double logarithm of the
discriminant and hence proving the existence of an absolute constant $c$ such that $\psi_\K(x)>0$
whenever $x\geq c \log^2 \disc_\K$. Later Oesterl\'{e}~\cite{Oesterle} announced that $c=70$ works
conditionally (see also~\cite[Th.~5]{Serre6}). More recently, Bach~\cite[Th.~4]{Bach2} proved (assuming
GRH, again) that the class group of $\K$ is generated by ideals whose norm is bounded by $12\log^2
\disc_\K$ and by $(4+o(1))\log^2 \disc_\K$ when $\disc_\K$ tends to infinity (see
also~\cite{BelabasDiazyDiazFriedman}). This proves the claim with $c=12$, and $c=4$ asymptotically. A
different approach of Bach and Sorenson~\cite{BachSorenson} proves that for any abelian extension of
number fields $\mathbb{E}/\K$ with $\mathbb{E}\neq \Q$ and every $\sigma\in\Gal(\mathbb{E},\K)$ there are
degree-one primes $\PP$ in $\K$ such that $\big[\frac{\mathbb{E}/\K}{\PP}\big]=\sigma$ with $\Norm\PP\leq
(1+o(1))\log^2 \disc_{\mathbb{E}}$, where the ``little-$o$'' function is explicit but decays very slowly.
As a consequence of the work of Lamzouri, Li and Soundararajan~\cite[Cor.~1.2]{LamzouriLiSoundararajan}
one can take $1+o(1)=\big(\tfrac{\eul(q)\log q} {\log \disc_\K}\big)^2$ in the case of the cyclotomic
extension $\K=\Q[q]$ of $q$-th roots of unity.\\
The case $\mathbb{E} = \K$ of the aforementioned result of Bach and Sorenson implies that there exists a
degree-one prime below $(1+o(1))\log^2 \disc_\K$. Using the bounds for $\psi^{(1)}_\K(x)$ and
$\psi^{(2)}_\K(x)$ in Theorem~\ref{th:B2} we reach a similar conclusion with the ``little-$o$'' function
substituted by an explicit and quite small constant.
\begin{corollary}\label{cor:B1}
(GRH) For every $\kappa\geq 0$, there are more than $\kappa$ degree-one prime ideals $\PP$ with
$\Norm\PP\leq \big(\mathcal{L}_\K + \sqrt{8\smash[b]{\kappa \log(\mathcal{L}_\K +
\sqrt[3]{\kappa}\log\kappa)}}\big)^2$, where $\mathcal{L}_\K:= 1.075(\log\disc_\K+13)$ (with
$\sqrt[3]{\kappa}\log\kappa=0$ for $\kappa=0$).
\end{corollary}

\begin{remark*}
The same argument, but this time based on bounds for $\psi^{(2)}_\K(x)$, $\psi^{(3)}_\K(x)$ and
$\psi^{(4)}_\K(x)$, produces a small improvement on the previous corollary, giving the same conclusion
but with $\mathcal{L}_\K:= 1.0578(\log\disc_\K + c)$ for a suitable constant $c$ which can be explicitly
computed. The improvement is due to the fact that the main constants $0.3526$ and $0.5375$ appearing in
Theorem~\ref{th:B2} satisfy $1.0578 = 3\cdot 0.3526 < 2\cdot 0.5375 = 1.075$. Actually, no further
improvement is possible with our technique (see Remark~\ref{rem:B2}). In our opinion this very small
improvement is unworthy of a detailed exposition: the interested reader will be able to prove it
following the proof of Corollary~\ref{cor:B1} in Section~\ref{sec:B4}.
\end{remark*}

Let $\partial_\K=\prod_\PP \PP^{c_\PP}$ be the decomposition of the different ideal of $\K$. We have
$c_\PP=e(\PP)-1$ when $\PP$ is tamely ramified and $c_\PP\geq e(\PP)$ when $\PP$ is wildly ramified. If
$\PP$ is above an odd prime then $\log\Norm\PP\geq \log 3$ hence $c_\PP\log\Norm\PP\geq \log 3$. If $\PP$
is above $2$, then either it is wildly ramified and $c_\PP\geq e(\PP)\geq 2$ or it is tamely ramified and
$c_\PP=e(\PP)-1\geq 2$ (by definition of tame ramification). We thus have $c_\PP\log\Norm\PP\geq \log 3$
in all cases. This in turn means that the number of ramifying ideals is at most
$\frac{\log\Norm\partial_\K}{\log 3}\leq \log\disc_\K$. We deduce immediately the following
\begin{corollary}\label{cor:B2}
(GRH) For every $\kappa\geq 0$, there are more than $\kappa$ unramified degree-one prime ideals $\PP$
with $\Norm\PP\leq \big(\mathcal{L}_\K + \sqrt{8\smash[b]{\kappa' \log(\mathcal{L}_\K +
\sqrt[3]{\kappa'}\log\kappa')}}\big)^2$, where $\kappa'=\kappa+\log\disc_\K$ and $\mathcal{L}_\K:= 1.075
(\log\disc_\K + 13)$.
\end{corollary}

\begin{remark*}
If $\K/\Q$ is a Galois extension, then the prime ideals in Corollary~\ref{cor:B2} are totally split, i.e.
$\big[\frac{\K/\Q}{\PP}\big] = \id$.
\end{remark*}
Let $\K:=\Q[q]$ be the cyclotomic field of $q$-th roots of unity. Let $p$ be the largest prime divisor of
$q$ and write $q=:p^\nu q'$ with $p$ and $q'$ coprime. There is a ramified prime ideal of degree one if
and only if $p \equiv 1 \pmod{q'}$, this condition being trivially true when $q'=1$, i.e. when $q$ is a
prime power. In that case there are $\varphi(q')$ ramified primes of degree one and their norm is $p$.
Therefore, there is necessarily a prime congruent to $1\pmod{q}$ below the bound of
Corollary~\ref{cor:B1} with $\kappa=\eul(q')$. A second prime congruent to $1$ modulo $q$ is produced
setting $\kappa = \eul(q')+\eul(q)$. Comparing $\mathcal{L}_\K$ and $\eul(q)\log q$ we get the following
explicit result.
\begin{corollary}\label{cor:B5}
(GRH) For every $q\geq 5$ there are at least two primes which are congruent to $1$ modulo $q$ and $\leq
1.2(\eul(q)\log q)^2$.
\end{corollary}
\begin{proof}
We know that $\log\disc_\K = \eul(q)\log q - \eul(q)\sum_{p|q}\frac{\log p}{p-1}$ (see~\cite[Prop.~2.17]
{Washington1}), so that $\mathcal{L}_\K\leq 1.075\eul(q)\log q$ for every $q>e^{13}$ (and when $q>32$ if
$q$ is not a prime).
%
%
Define $q=:p^\nu q'$ as above. As observed, we take $\kappa=\eul(q')+\eul(q)$ in Corollary~\ref{cor:B1}.\\
Notice that, if $q'\neq 1$, then $p\geq 3$ thus
$\eul(q')=\frac{\eul(q)}{(p-1)p^{\nu-1}}\leq\frac{1}{2}\eul(q)$, while if $q'=1$ the same inequality
holds as soon as $q\geq 3$. This proves that $\kappa\leq\frac{3}{2}\eul(q)$ holds for every $q\geq 3$.\\
Since $\log\disc_\K\geq \frac{1}{2}\eul(q)\log q$ for $q\geq 7$, one has
$\eul(q)\leq 4\mathcal{L}_\K/\log(2\mathcal{L}_\K)$. Thus
$\kappa \leq 6\mathcal{L}_\K/\log(2\mathcal{L}_\K)$ when
$q\geq 7$. With this upper bound, for $\mathcal{L}_\K\geq 1.3\cdot10^5$, we get
\[
1.075^2\cdot\Big(1 + \frac{1}{\mathcal{L}_\K}\sqrt{8\smash[b]{\kappa \log(\mathcal{L}_\K + \sqrt[3]{\kappa}\log\kappa)}}\Big)^2
\leq 1.2.
\]
If $\eul(q)\geq 24000$, we have $\mathcal{L}_\K\geq 1.075(\frac{1}{2}\eul(q)\log\eul(q)+13)\geq
1.3\cdot10^5$. For $q\geq 510510=2\cdot 3\cdot 5\cdot 7\cdot 11\cdot 13\cdot 17$, looking separately the
cases where $q$ as at least $7$ or less than $7$ distinct prime factors, we see that $\eul(q)\geq
92160\geq 24000$. This proves the claim for $q\geq 510510$.
%
%
%
Then, the explicit computation for $q<510510$ of the bound in Corollary~\ref{cor:B1} shows that it
is $\leq 1.2(\eul(q)\log q)^2$ for every $q>4373$; this proves the claim for $4373<q<510510$. A direct
search shows that two primes $p=1\pmod{q}$ and $p\leq 1.2(\eul(q)\log q)^2$ exist also in the range
$5\leq q\leq 4373$.
%
%
\end{proof}
\begin{remark*}
We can repeat the proof of the previous corollary in a more general setting. Letting
$\kappa=\eul(q')+(k-1)\eul(q)$ one can prove that, when $q>e^{13}$, there are at least $k$ primes
congruent to $1$ modulo $q$ and smaller than
\[
\big((1.075+0.02\sqrt{k\log k})\eul(q)\log q\big)^2.
\]
%
\end{remark*}

\subsection*{Computing the residue of $\zeta_\K$}
An explicit form for the remainder of the formula for any $\psi^{(m)}_\K$ gives a way to compute within a
prefixed error any quantity which can be written as a Dirichlet series in the von Mangoldt function of
the field. Among these, the computation of the logarithm of the residue of $\zeta_\K$ with an error lower
than $\tfrac{1}{2}\log 2$ is a particularly important problem, being an essential step of Buchmann's
algorithm~\cite{Buchmann} for the computation of the class group and the regulator of the ring of
integral elements in $\K$. The representation
\[
\log\zeta_\K(s)-\log\zeta(s)
= \sum_{n=2}^\infty \frac{\tilde{\Lambda}_\K(n)-\Lambda_\Q(n)}{n^s\log n}
\]
holds true uniformly in $\Ree(s)\geq 1$ by Landau's and de la Vall\'{e}e--Poussin's estimates for the
remainder terms of $\psi_\K(x)$ and $\psi_\Q(x)$. Hence, a simple way to compute the residue is
\[
\log\underset{s=1}{\res}\zeta_\K(s)
= \lim_{s\to 1}[\log\zeta_\K(s)-\log\zeta(s)]
= \sum_{n=2}^\infty \frac{\tilde{\Lambda}_\K(n)-\Lambda_\Q(n)}{n\log n}.
\]
Here, truncating the series at a level $N$ and using the partial summation formula one gets
\begin{equation}\label{eq:B3}
\log\underset{s=1}{\res}\zeta_\K(s)\\
= \sum_{n\leq N} \big(\tilde{\Lambda}_\K(n)-\Lambda_\Q(n)\big)\big(f(n)-f(N)\big) + \RR(N)
\end{equation}
with $f(x):=(x\log x)^{-1}$ and
\begin{align*}
\RR(N)
:= -\int_N^{+\infty} (\psi_\K(x)-\psi_\Q(x)) f'(x)\dd x.
\end{align*}
Moving the absolute value into the integral and using~\eqref{eq:B0} yields
\begin{align*}
|\RR(N)|
\leq \frac{c}{\sqrt{N}}(\log\disc_\K + n_\K\log N)
\end{align*}
for an explicit constant $c$. This procedure can already be used to compute the residue, but a
substantial improvement has been obtained by Bach~\cite{Bach} and very recently announced by Belabas and
Friedman~\cite{BelabasFriedman}. They propose different approximations to $\log\res_{s=1} \zeta_\K(s)$
with a remainder term which is essentially estimated by $c\frac{\log\disc_\K}{\sqrt{N}\log N}$, with
$c=8.33$ in Bach's work and $c=2.33$ in the one of Belabas and Friedman. The presence of the extra $\log
N$ in the denominator and the small multiplicative constant in their formulas represent a strong boost to
the computation, but this is achieved at the cost of some complexities in the proofs and in the
implementation of the algorithm.\\
Using Theorem~\ref{th:B2} after a further integration by parts of Equation~\eqref{eq:B3} we get the same
result with a simpler approach and already smaller constants. Even stronger results are available in
Section~\ref{sec:B5}. The following corollary is a part of Corollary~\ref{cor:B3}.
\begin{corollary}\label{cor:Bintro}
(GRH) For $N\geq 3$, we have
\[
\log\underset{s=1}{\res}\zeta_\K(s)
= \sum_{n\leq N} \big(\tilde{\Lambda}_\K(n)-\Lambda_\Q(n)\big) \big(f(n) - f(N) - (n-N)f'(N)\big) + \RR^{(1)}(N)
\]
with
\begin{gather*}
|\RR^{(1)}(N)| \leq  \alpha^{(1)}_\K \Big(\frac{\frac{5}{2}+y}{\sqrt{N}\log N} + \frac{3}{4} \ExpInt\Big(\frac{1}{2}\log N\Big)\Big)
                   + \beta ^{(1)}_\K \frac{2+3y}{N}
                   + \gamma^{(1)}_\K \frac{2y+y^2}{N}
                   + \delta^{(1)}_\K \frac{y+y^2}{N^2},
\end{gather*}
$f(x):=(x\log x)^{-1}$, $y:=(\log N)^{-1}$, $E_1(x):= \int_1^{+\infty} e^{-xt}t^{-1}\dd t$ and
\begin{align*}
   \alpha^{(1)}_\K &= 0.5375\log\disc_\K - 1.0355 n_\K + 5.4341,
&&&\beta ^{(1)}_\K &= n_\K-1,                                      \\
   \gamma^{(1)}_\K &= 1.0155\log\disc_\K - 2.1041 n_\K + 10.1799,
&&&\delta^{(1)}_\K &=       \log\disc_\K - 1.415  n_\K + 4.
\end{align*}
\end{corollary}
\noindent
The $\ExpInt$ function satisfies the double inequality $1-1/x\leq xe^x\ExpInt(x)\leq 1$ for every $x>0$.
Thus this strategy produces an error bounded essentially by $2.15 \frac{\log\disc_\K}{\sqrt{N}\log N}$:
this means that our algorithm is in $N$ of the same order of Bach's and Belabas--Friedman's results with
a smaller constant. Moreover, the negative coefficient for the contribution of the degree has the
interesting side effect that, for fixed discriminant, the complexity actually decreases for increasing
degree.

As shown in Tables~\ref{tab:B3} and \ref{tab:B4} below, in practice Corollary~\ref{cor:Bintro} improves
on Belabas and Friedman's procedure by a factor of about $3$, and in some ranges even by a factor of $10$.

\newpage
\section{Preliminary inequalities}\label{sec:B2}
For $\Ree(s) > 1$ we have
\[
-\frac{\zeta_\K'}{\zeta_\K}(s)
= \sum_\PP  \sum_{m=1}^\infty \log(\Norm\PP)(\Norm\PP)^{-ms},
\]
which in terms of standard Dirichlet series reads
\[
-\frac{\zeta_\K'}{\zeta_\K}(s)
= \sum_{n=1}^{\infty} \tilde{\Lambda}_\K(n)n^{-s}
\quad
\text{observing that}
\quad
\tilde{\Lambda}_\K(n)
=
\begin{cases}
\displaystyle
\sum_{\PP|p,\, f_\PP|k} \log \Norm\PP & \text{if $n=p^k$}\\
\quad\ 0                              & \text{otherwise},
\end{cases}
\]
where $f_\PP$ is the residual degree of $\PP$. The formula for $\tilde{\Lambda}_\K$ shows that
$\tilde{\Lambda}_\K(n)\leq n_\K \Lambda(n)$ for every integer $n$, so that immediately we get
\begin{equation}\label{eq:B11}
0<-\frac{\zeta_\K'}{\zeta_\K}(\sigma)\leq -n_\K\frac{\zeta'}{\zeta}(\sigma)
\qquad
\forall\sigma>1.
\end{equation}
Let
\begin{equation}\label{eq:B12}
\Gamma_\K(s) := \Big[\pi^{-\frac{s+1}{2}}\Gamma\Big(\frac{s+1}{2}\Big)\Big]^{r_2}
                \Big[\pi^{-\frac{s}{2}}  \Gamma\Big(\frac{s}{2}\Big)  \Big]^{r_1+r_2}
\end{equation}
and
\begin{equation}\label{eq:B13}
\xi_\K(s) := s(s-1) \disc_\K^{s/2}\Gamma_\K(s)\zeta_\K(s).
\end{equation}
The functional equation for $\zeta_\K$ then reads
\begin{equation}\label{eq:B14}
\xi_\K(1-s) = \xi_\K(s).
\end{equation}
Since $\xi_\K(s)$ is an entire function of order $1$ and does not vanish at $s = 0$, one has
\begin{equation}\label{eq:B15}
\xi_\K(s) = e^{A_\K+B_\K s} \prod_\rho \Big(1 - \frac{s}{\rho}\Big) e^{s/\rho}
\end{equation}
for some constants $A_\K$ and $B_\K$, where $\rho$ runs through all the zeros of $\xi_\K(s)$. These are
precisely the zeros $\rho = \beta + i\gamma$ of $\zeta_\K(s)$ for which $0 < \beta < 1$ and are the
so-called ``nontrivial zeros'' of $\zeta_\K(s)$. From now on $\rho$ will denote a nontrivial zero of
$\zeta_\K(s)$. We recall that the zeros are symmetric with respect to the real axis, as a consequence of
the fact that $\zeta_\K(s)$ is real for $s\in\R$.\\
Differentiating~\eqref{eq:B13} and~\eqref{eq:B15} logarithmically we obtain the identity
\begin{equation}\label{eq:B16}
\frac{\zeta'_\K}{\zeta_\K}(s)
  = B_\K
   + \sum_\rho \Big(\frac{1}{s-\rho}+\frac{1}{\rho}\Big)
   - \frac{1}{2}\log \disc_\K
   - \Big(\frac{1}{s}+\frac{1}{s-1}\Big)
   - \frac{\Gamma'_\K}{\Gamma_\K}(s).
\end{equation}
Stark~\cite[Lemma~1]{Stark1} proved that the functional equation~\eqref{eq:B14} implies that $B_\K \!=
-\!\sum_\rho \Ree(\rho^{-1})$ (see also~\cite{Odlyzko5} and~\cite[Ch.~XVII, Th.~3.2]{Lang1}), and that once
this information is available one can use~\eqref{eq:B16} and the definition of the gamma factor
in~\eqref{eq:B12} to prove that the function $f_\K(s) := \sum_\rho\Ree\big(\frac{2}{s-\rho}\big)$ can be
computed via the alternative representation
\begin{equation}\label{eq:B17}
f_\K(s)
= 2\Ree\frac{\zeta'_\K}{\zeta_\K}(s)
    +\log\frac{\disc_\K}{\pi^{n_\K}}
    +\Ree\Big(\frac{2}{s}
             +\frac{2}{s-1}
        \Big)
    +(r_1+r_2)\Ree\frac{\Gamma'}{\Gamma}\Big(\frac{s}{2}\Big)
    +r_2\Ree\frac{\Gamma'}{\Gamma}\Big(\frac{s+1}{2}\Big).
\end{equation}
Using~\eqref{eq:B13}, \eqref{eq:B14} and~\eqref{eq:B16} one sees that
\begin{equation}\label{eq:B18}
\frac{\zeta'_\K}{\zeta_\K}(s) =
\begin{cases}
 \displaystyle\frac{r_1+r_2-1}{s} + r_\K + O(s)      & \text{as $s\to  0$}\\[1.5ex]
 \displaystyle\frac{r_2}{s+1} + r'_\K + O(s+1)       & \text{as $s\to -1$}\\[1.5ex]
 \displaystyle\frac{r_1+r_2}{s+2} + r''_\K + O(s+2)  & \text{as $s\to -2$},
\end{cases}
\end{equation}
where
\begin{subequations}\label{eq:AB}
\begin{align}
r_\K  &= B_\K +1
        -\frac{1}{2} \log\frac{\disc_\K}{\pi^{n_\K}}
        -\frac{r_1+r_2}{2}\frac{\Gamma'}\Gamma(1)
        -\frac{r_2}{2}    \frac{\Gamma'}\Gamma\Big(\frac{1}{2}\Big),   \label{eq:Bbis}\\
r'_\K  &= -\frac{\zeta'_\K}{\zeta_\K}(2)
         -\log\frac{\disc_\K}{\pi^{n_\K}}
        -\frac{n_\K}{2}\frac{\Gamma'}\Gamma\Big(\frac{3}{2}\Big)
        -\frac{n_\K}{2}\frac{\Gamma'}\Gamma(1),                        \label{eq:B19}\\
r''_\K &= -\frac{\zeta'_\K}{\zeta_\K}(3)
         -\log\frac{\disc_\K}{\pi^{n_\K}}
        -\frac{n_\K}{2}\frac{\Gamma'}\Gamma(2)
        -\frac{n_\K}{2}\frac{\Gamma'}\Gamma\Big(\frac{3}{2}\Big).      \label{eq:B20}
\end{align}
\end{subequations}
In order to prove our results we need explicit bounds for $B_\K$, $r_\K$, $r'_\K$ and $r''_\K$ and for
some auxiliary functions.
\begin{lemma}\label{lem:A4}
$B_\K$ is real, negative, and under GRH we have
\[
|B_\K| \leq 0.5155\log\disc_\K - 1.2432 n_\K + 9.3419.
\]
\end{lemma}
\begin{proof}
We know that $-B_\K = \sum_\rho \Ree\big(\frac{1}{\rho}\big) = \sum_\rho\tfrac{\Ree(\rho)}{|\rho|^2}$,
which is positive. The upper bound will be proved in next section.
\end{proof}

\begin{lemma}\label{lem:B1}
(GRH) We have
\begin{align*}
|r_\K|  &\leq  1.0155\log \disc_\K - 2.1042 n_\K + 8.3423, \\
|r'_\K| &\leq  \log \disc_\K - 1.415 n_\K + 4,             \\
|r''_\K|&\leq  \log \disc_\K - 0.9151 n_\K + 2.
\end{align*}
\end{lemma}
\begin{proof}
Substituting the values $-\frac{\Gamma'}{\Gamma}(\tfrac{1}{2}) = \gamma + 2\log 2$,
$-\frac{\Gamma'}{\Gamma}(1) = \gamma$ in~\eqref{eq:Bbis} we get
\begin{equation}\label{eq:A17}
r_\K  = B_\K
       -\frac{1}{2} \log \disc_\K
       +(\log\pi+\gamma)\frac{n_\K}{2}
       +r_2 \log 2
       +1.
\end{equation}
By Lemma~\ref{lem:A4} we get
\begin{align*}
r_\K&\leq -\frac{1}{2}\log \disc_\K + (\gamma+\log 2\pi)\frac{n_\K}{2} +1
     \leq -\frac{1}{2}\log \disc_\K + 1.2076 n_\K +1
\intertext{and}
r_\K
    &\geq  -(0.5155\log\disc_\K - 1.2432 n_\K + 9.3423) -\frac{1}{2} \log \disc_\K +(\log\pi+\gamma)\frac{n_\K}{2} +1\\
    &\geq -1.0155\log \disc_\K +2.1042 n_\K - 8.3423.
\end{align*}
The (opposite of the) lower bound for $r_\K$ gives the upper bound for $|r_\K|$, since the explicit
bounds for the discriminant in terms of the degree proved by Odlyzko
(see~\cite{Odlyzko1,Odlyzko2,Odlyzko3,Odlyzko4} and Table~3 in~\cite{OdlyzkoTables}) show that the
difference
\begin{multline}\label{eq:A18}
1.0155\log \disc_\K -2.1042 n_\K + 8.3423 -\big(-\tfrac{1}{2}\log \disc_\K + 1.2076 n_\K+1\big)\\
= 1.5155 \log \disc_\K - 3.3118 n_\K + 7.3423
\end{multline}
is always positive (use the entry $b=1.3$ in~\cite[Tab.~3]{OdlyzkoTables}).\\
The bounds for $r'_\K$ and $r''_\K$ are proved with a similar argument. By~\eqref{eq:B19} and the
identities $-\frac{\Gamma'}{\Gamma}(\tfrac{3}{2}) = \gamma+2\log 2-2$, $-\frac{\Gamma'}{\Gamma}(1) =
\gamma$ we have
\begin{align*}
r'_\K& =   - \frac{\zeta'_\K}{\zeta_\K}(2)
           - \log \disc_\K
           + (\log2\pi +\gamma -1)n_\K.
\intertext{By~\eqref{eq:B11} we have}
r'_\K&\leq -\log \disc_\K + \Big(-\frac{\zeta'}{\zeta}(2) + \log2\pi +\gamma -1\Big)n_\K
      \leq -\log \disc_\K + 1.9851 n_\K
%
\intertext{and}
r'_\K&\geq -\log \disc_\K + (\log2\pi +\gamma -1)n_\K
      \geq -\log \disc_\K + 1.415 n_\K
      \geq -\log \disc_\K + 1.415 n_\K - 4.
\end{align*}
The lower bounds for the discriminant prove that the inequality
\begin{equation}\label{eq:B21}
\log \disc_\K - 1.415 n_\K + 4  -\big(-\log \disc_\K + 1.9851 n_\K\big)
= 2 \log \disc_\K - 3.4001 n_\K + 4\geq 0
\end{equation}
is true for $n_\K\geq 5$ (entry $b=1$ in~\cite[Tab.~3]{OdlyzkoTables}). Using the ``megrez'' number field
tables~\cite{MegrezTables} we find that~\eqref{eq:B21} has only two exceptions for fields of equation
$x^2+x+1$ and $x^4-x^3-x^2+x+1$. We numerically compute the value of $r'_\K$ for these two fields and we
find that indeed $|r'_\K|\leq \log\disc_\K - 1.415 n_\K + 4$.\\
\hbox to\hsize{\hss}      
\noindent
Finally, by~\eqref{eq:B20}
\[
r''_\K = - \frac{\zeta'_\K}{\zeta_\K}(3)
         - \log \disc_\K
         + \Big(\log2\pi +\gamma - \frac{3}{2}\Big)n_\K
\]
and thus
\begin{align*}
r''_\K&\leq -\log\disc_\K + \Big(-\frac{\zeta'}{\zeta}(3)+\log2\pi+\gamma - \frac{3}{2}\Big)n_\K
       \leq -\log\disc_\K + 1.08 n_\K
\intertext{and}
r''_\K&\geq -\log\disc_\K + (\log2\pi+\gamma-\tfrac{3}{2})n_\K
       \geq -\log\disc_\K + 0.9151 n_\K
       \geq -\log\disc_\K + 0.9151 n_\K - 2.
\end{align*}
The lower bounds for the discriminant prove that the inequality
\begin{equation}\label{eq:B22}
\log\disc_\K - 0.9151n_\K + 2 -\big(-\log\disc_\K + 1.08 n_\K\big)
= 2 \log \disc_\K - 1.9951 n_\K + 2 \geq 0
\end{equation}
is true for all $n_\K$ (entry $b=0.6$ in~\cite[Tab.~3]{OdlyzkoTables}).
\end{proof}

\begin{lemma}\label{lem:B2}
For $x\geq 1$ let
\[
\begin{array}{l c l}
\displaystyle
f^{(1)}_1(x) := \sum_{r=1}^{\infty}\frac{x^{1-2r}}{2r(2r-1)},
&
\displaystyle
f^{(1)}_2(x) := \sum_{r=2}^{\infty}\frac{x^{2-2r}}{(2r-1)(2r-2)},\\
\displaystyle
f^{(2)}_1(x) := \sum_{r=2}^\infty \frac{x^{2-2r}}{r(2r-1)(2r-2)},
&
\displaystyle
f^{(2)}_2(x) := \sum_{r=1}^\infty \frac{x^{1-2r}}{(2r+1)r(2r-1)},
\end{array}
\]
and
\enlargethispage{\baselineskip}
\begin{align*}
R^{(1)}_{r_1,r_2}(x) :=& - d_\K x(\log x-1)
                         + r_2(\log x +1)
                         - (r_1+r_2)f^{(1)}_1(x)
                         - r_2      f^{(1)}_2(x),                 \\
R^{(2)}_{r_1,r_2}(x) :=& - d_\K x^2\Big(\log x -\frac{3}{2}\Big)
                         + 2r_2 x\log x
                         - (r_1+r_2)\Big(\log x +\frac{3}{2}\Big) \\
                       & + (r_1+r_2)f^{(2)}_1(x)
                         + r_2      f^{(2)}_2(x).
\end{align*}
If $x\geq 3$ then
\begin{align*}
|R^{(1)}_{r_1,r_2}(x)| &\leq (n_\K - 1)x\log x + \delta_{n_\K,1}\frac{0.5097}{x},                          \\
|R^{(2)}_{r_1,r_2}(x)| &\leq (n_\K-1)x^2(\log x - \tfrac{1}{2}) + \delta_{n_\K,1}(\log x +2)
\end{align*}
where $\delta_{n_\K,1}$ is $1$ if $n_\K=1$ and $0$ otherwise.
\end{lemma}
\begin{proof}
We have
\begin{align*}
f^{(1)}_1(x) &= \frac{1}{2}\Big[x\log(1-x^{-2}) + \log\Big(\frac{1+x^{-1}}{1-x^{-1}}\Big)\Big],            \\
f^{(1)}_2(x) &= 1 - \frac{1}{2}\Big[\log(1-x^{-2}) + x\log\Big(\frac{1+x^{-1}}{1-x^{-1}}\Big)\Big],        \\
f^{(2)}_1(x) &= \frac{3}{2} - \frac{1}{2}(x^2+1)\log(1-x^{-2}) - x\log\Big(\frac{1+x^{-1}}{1-x^{-1}}\Big), \\
f^{(2)}_2(x) &=  - x + x\log(1-x^{-2}) + \frac{1}{2}(x^2+1)\log\Big(\frac{1+x^{-1}}{1-x^{-1}}\Big),
\end{align*}
and the claims follow with elementary arguments.
%
%
%
\end{proof}

\section{proof of the theorem}\label{sec:B3}
When $m\geq 1$ the equality in~\eqref{eq:D1} follows by the Dirichlet series representation of
$\frac{\zeta'_\K}{\zeta_\K}(s)$ and the special integrals
\[
\frac{m!}{2\pi i}\int_{2-i\infty}^{2+i\infty} \frac{y^{s+m}}{\prod_{u=0}^m(s+u)}\dd s
= \begin{cases}
  (y-1)^m   &\text{if $y > 1$}        \\
  0         &\text{if $0 < y \leq 1$}
  \end{cases}
\qquad
\forall m\geq 1.
\]
The case $m=0$ is more complicated but well known (see~\cite{LagariasOdlyzko}).
Equalities~(\ref{eq:E1a}--\ref{eq:E1b}) come from the Cauchy residue theorem, using the identities
\begin{align*}
\frac{x^{s+1}}{s(s+1)}
&=
\begin{cases}
\vbox{\hbox{\hphantom{$\frac{1}{2(s+2)} +  \frac{1}{2}\log x +\frac{3}{4} + O(s+2)$}}
      \hbox{$ \frac{x}{s}   + x\log x - x + O(s)$}}            & \text{as $s\to 0 $}\\
             -\frac{x}{s+1} -  \log x - 1 + O(s+1)             & \text{as $s\to -1$},
\end{cases}\\
\frac{x^{s+2}}{s(s+1)(s+2)}
&=
\begin{cases}
\frac{x^2}{2s}   +  \frac{x^2}{2}\log x -\frac{3}{4}x^2 + O(s) & \text{as $s\to 0 $}\\
-\frac{x}{s+1}   - x\log x + O(s+1)                            & \text{as $s\to -1$}\\
\frac{1}{2(s+2)} +  \frac{1}{2}\log x +\frac{3}{4} + O(s+2)    & \text{as $s\to -2$},
\end{cases}
\end{align*}
and the definitions of $r_\K$, $r'_\K$ and $r''_\K$ in~\eqref{eq:B18} and of $R^{(m)}_{r_1,r_2}(x)$
in Lemma~\ref{lem:B2}. They show that
\begin{align*}
\Big|\psi^{(1)}_\K(x) - \frac{x^2}{2}\Big|
&\leq x^{3/2}\sum_{\rho}\frac{1}{|\rho(\rho+1)|}
            + |x r_\K
               - r'_\K
               - R^{(1)}_{r_1,r_2}(x)|,           \\
\Big|\psi^{(2)}_\K(x) - \frac{x^3}{3}\Big|
&\leq  x^{5/2}\sum_{\rho} \frac{2}{|\rho(\rho+1)(\rho+2)|}
     + |x^2 r_\K
        - 2x r'_\K
        +    r''_\K
        - R^{(2)}_{r_1,r_2}(x)|.
\end{align*}
For $\Q$, we observe that $|x r_\Q - r'_\Q - R^{(1)}_{1,0}(x)|\leq x\log 2\pi$ and $|x^2 r_\Q - 2xr'_\Q +
r''_\Q - R^{(2)}_{1,0}(x)|$ $\leq x^2\log 2\pi$.
%
%
%
%
%
%
%
For generic $\K$ these terms are estimated with the sums of the absolute values, and $|r_\K|$, $|r'_\K|$,
$|r''_\K|$ and $|R^{(j)}_{r_1,r_2}(x)|$ have already been estimated in Lemmas~\ref{lem:B1}
and~\ref{lem:B2}. We thus only need a bound for $\sum_{\rho}|\rho(\rho+1)|^{-1}$ and
$\sum_{\rho}|\rho(\rho+1)(\rho+2)|^{-1}$. It is easy to check that
\begin{equation}\label{eq:E2}
\sum_{\rho}\frac{1}{|\rho(\rho+1)|}
\leq \frac{2}{3}f_\K\Big(\frac{3}{2}\Big)
\qquad\text{and}\qquad
\sum_{\rho}\frac{1}{|\rho(\rho+1)(\rho+2)|}
\leq \frac{4}{15}f_\K\Big(\frac{3}{2}\Big).
\end{equation}
A bound comes from the estimation $f_\K\big(\frac{3}{2}\big) \leq \log \disc_\K  - (\gamma+\log 8\pi-2)
n_\K + \frac{16}{3}$, which is the case $a=1/2$ of Lemma~5.6 in~\cite{Bach2} and of Lemma~4.6
in~\cite{BachSorenson}, but we can do better.
\begin{lemma}\label{lem:B3}
(GRH) We have
\begin{align*}
\sum_{\rho}\frac{1}{|\rho(\rho+1)|}
&\leq 0.5375\log\disc_\K - 1.0355n_\K + 5.3879,\\
\sum_{\rho}\frac{1}{|\rho(\rho+1)(\rho+2)|}
&\leq 0.1763\log\disc_\K - 0.4106 n_\K + 2.2496.
\end{align*}
For the Riemann zeta function the conclusions improve to
\[
\sum_{\rho}\frac{1}{|\rho(\rho+1)|}
\leq 0.0462,
\qquad
\sum_{\rho}\frac{1}{|\rho(\rho+1)(\rho+2)|}
\leq 0.00146.
\]
\end{lemma}
\begin{proof}
We apply the method we have described in the introduction with real $s$, so
$f(s,\gamma)=4(2s-1)/((2s-1)^2+4\gamma^2)$. We choose
$D=\{s_j\colon j=1,2,\dots\}$ with $s_j:=1+j/2$, and $\mu$ compactly supported on $D$. For the first
claim let $g(\gamma):= 4/((1+4\gamma^2)(9+4\gamma^2))^{1/2}$, so that $\sum_{\rho}|\rho(\rho+1)|^{-1} =
\sum_{\gamma} g(\gamma)$. Condition~\eqref{eq:BB23} indicates that we must prove
\begin{equation}\label{eq:B23}
g(\gamma) \leq F(\gamma) := \sum_j a_j f(s_j,\gamma)
\qquad \forall\gamma\in\R
\end{equation}
for suitable $a_j$. Recalling that $f_\K(s) = \sum_{\gamma} f(s,\gamma)$, Inequality~\eqref{eq:BB24}
gives
\begin{equation}\label{eq:B24}
\sum_{\rho}\frac{1}{|\rho(\rho+1)|} \leq \sum_j a_j f_\K(s_j),
\end{equation}
which generalizes~\eqref{eq:E2}. From~\eqref{eq:B24} and~\eqref{eq:B17}, and once~\eqref{eq:B23} is
proved, we obtain a bound for $\sum_{\rho} |\rho(\rho+1)|^{-1}$. The final coefficient of $\log\disc_\K$
will then be the sum of all $a_j$, thus we are interested in linear combinations for which this sum is as
small as possible. We choose the support of $\mu$ such that the $s_j$ appearing in~\eqref{eq:B23} are
those with $1\leq j\leq 2q$ for a suitable integer $q$. Let $\Upsilon \subset (0,\infty)$ be a set with
$q-1$ numbers. We require:
\begin{enumerate}
\item $g(\gamma)=F(\gamma)$ for all $\gamma\in\{0\}\cup\Upsilon$,
\item $g'(\gamma)=F'(\gamma)$ for all $\gamma\in\Upsilon$,
\item $\lim_{\gamma\to\infty}\gamma^2 g(\gamma)=\lim_{\gamma\to\infty}\gamma^2 F(\gamma)$.
\end{enumerate}
This produces a set of $2q$ linear equations for the $2q$ constants $a_j$. The first conditions impose a
double contact between $g$ and $F$ in all the points of $\Upsilon$. This means that $g$  will almost
certainly not cross $F$ at these points. With a little bit of luck, $F$ will be always above $g$
ensuring~\eqref{eq:B23}. We chose $q:=40$ and $\Upsilon:=\{v^i-v+1\colon 1\leq i\leq q-1\}$ for $v:=1.21$.
Finally, with an abuse of notation we took for $a_j$ the solution of the system, rounded above to
$10^{-7}$: this produces the numbers in Table~\ref{tab:Bbho1}. Then, using Sturm's algorithm, we prove
that the values found actually give an upper bound for $g$, so that~\eqref{eq:B24} holds with such
$a_j$'s. These constants verify
\begin{equation}\label{eq:B25}
\begin{aligned}
&\sum_j a_j                                                =     0.53747\ldots,&\qquad&\sum_j a_j\Big(\frac{2}{s_j}+\frac{2}{s_j-1}\Big)         \leq  5.3879,\\
&\sum_j a_j\frac{\Gamma'}{\Gamma}\Big(\frac{s_j}{2}\Big)   \leq -0.6838,       &\qquad&\sum_j a_j\frac{\Gamma'}{\Gamma}\Big(\frac{s_j+1}{2}\Big) \leq -0.1567.
\end{aligned}
\end{equation}
Moreover, the sum $\sum_j a_j\frac{\zeta'_\K}{\zeta_\K}(s_j)$ is negative. Indeed we write it as
\[
-\sum_n\tilde\Lambda_\K(n)S(n)\quad\text{with}\quad S(n):=\sum_j \frac{a_j}{n^{s_j}}
\]
and, since the signs of the $a_j$'s alternate, we can easily prove that the sum in pairs $\frac{a_1}{n^{s_1}}
+ \frac{a_2}{n^{s_2}}$, \ldots, $\frac{a_{2q-1}}{n^{s_{2q-1}}} + \frac{a_{2q}}{n^{s_{2q}}}$ are positive
for $n\geq 26500$. Then we check numerically that $S(n)>0$ also for $n\leq 26500$. The result now follows
from~\eqref{eq:B17}, \eqref{eq:B24} and~\eqref{eq:B25}.
\medskip\\
For the second inequality, let $g(\gamma):= 8/((1+4\gamma^2)(9+4\gamma^2)(25+4\gamma^2))^{1/2}$, so that
$\sum_{\rho}|\rho(\rho+1)(\rho+2)|^{-1} = \sum_{\gamma} g(\gamma)$. We use $s_j$ with $1\leq j\leq 2q-1$,
$q:=20$, $\Upsilon:=\{v^i-v+0.75\colon 1\leq i\leq q-1\}$, keeping $v=1.21$, and the conditions
\begin{enumerate}
\item $g(\gamma)=F(\gamma)$ for all $\gamma\in\{0\}\cup\Upsilon$,
\item $g'(\gamma)=F'(\gamma)$ for all $\gamma\in\Upsilon$.
\end{enumerate}
We take for $a_j$ the solution of the system, rounded above to $10^{-7}$: this produces the numbers in
Table~\ref{tab:Bbho2}. We check their validity using Sturm's algorithm as before. We then have
\begin{equation}\label{eq:B27}
\sum_{\rho}\frac{1}{|\rho(\rho+1)(\rho+2)|} \leq \sum_j a_j f_\K(s_j)
\end{equation}
where the constants $a_j$ verify
\begin{equation}\label{eq:B28}
\begin{aligned}
&\sum_j a_j                                                =     0.17629\ldots,&\qquad&\sum_j a_j\Big(\frac{2}{s_j}+\frac{2}{s_j-1}\Big)         \leq  2.2496,\\
&\sum_j a_j\frac{\Gamma'}{\Gamma}\Big(\frac{s_j}{2}\Big)   \leq -0.3130,       &\qquad&\sum_j a_j\frac{\Gamma'}{\Gamma}\Big(\frac{s_j+1}{2}\Big) \leq -0.1047,\\
&\sum_j a_j\frac{\zeta'_\K}{\zeta_\K}(s_j) \leq  0.
\end{aligned}
\end{equation}
As before, we prove the last inequality noticing that it is $-\sum_n\tilde\Lambda_\K(n) S(n)$ with
$S(n):=\sum_j \frac{a_j}{n^{s_j}}$, and that each $S(n)$ is positive since this is true for $n\leq 16800$
(numerical test) and since the sums in pairs $\frac{a_1}{n^{s_1}} + \frac{a_2}{n^{s_2}}$, \ldots,
$\frac{a_{2q-3}}{n^{s_{2q-3}}} + \frac{a_{2q-2}}{n^{s_{2q-2}}}$ and the last summand
$\frac{a_{2q-1}}{n^{s_{2q-1}}}$ are positive for
$n\geq 16800$. The result now follows from~\eqref{eq:B17}, \eqref{eq:B27} and~\eqref{eq:B28}.\\
For the Riemann zeta function we proceed as in the general case, but now using the numerical value of
$\sum_j a_j f_\Q(s_j)$.
\end{proof}

\begin{remark}\label{rem:B1}
For the Riemann zeta function one has $\sum_{|\gamma|\geq T}|\rho|^{-2} \leq 10^{-5}$ when $T\geq 400000$
(by partial summation, using~\cite[Th.~19]{Rosser3} or \cite[Cor.~1]{TrudgianII}), thus the value of
$\sum_{\rho}|\rho(\rho+1)|^{-1}$ correct up to the fifth digit can be obtained summing the first $7\cdot
10^5$ zeros. The computation produces the number $0.0461(1)$.
%
In a similar way, $\tfrac{1}{T}\sum_{|\gamma|\leq T}|\rho|^{-2} \leq 10^{-10}$ when $T\geq 200000$, thus
the value of $\sum_{\rho}|\rho(\rho+1)(\rho+2)|^{-1}$ correct up to the tenth digit can be obtained
summing the first $3\cdot 10^5$ zeros. The computation produces the number $0.001439963(2)$. In both
cases the bounds in Lemma~\ref{lem:B3} essentially agree with the actual values.
\end{remark}

\begin{remark}\label{rem:B2}
Let $g_m(\gamma):=\prod_{n=0}^m|n+\tfrac{1}{2}+i\gamma|^{-1}$. As observed in Remark~\ref{rem:Bintro},
$\sum_j a_j\geq \frac{1}{2\pi}\int_\R g_1(\gamma)\dd\gamma\geq 0.53659$ in the first case, and $\sum_j
a_j\geq \frac{1}{2\pi}\int_\R g_2(\gamma)\dd\gamma\geq 0.1759$ in the second case are the best
coefficients of $\log\disc_\K$ we can get from our method. Thus, what we got in Lemma~\ref{lem:B3} are
close to the best. Moreover, for a generic $m\geq 1$ one gets
\[
\Big|\psi^{(m)}_\K(x) - \frac{x^{m+1}}{m+1}\Big|
\leq
m!\,x^{m+1/2} \sum_\rho g_m(\gamma) + \text{lower order terms}
\]
and we need an upper bound of $\sum_\rho g_m(\gamma)$. If we could follow the argument proving
Lemma~\ref{lem:B3} for general $m$ we would get a sequence $a_j$ (a different sequence for every $m$)
%
necessarily satisfying the lower bound $\sum_j a_j\geq \frac{1}{2\pi} \int_\R g_m(\gamma)\dd\gamma$.
Since $\int_\R g_m(\gamma)\dd\gamma \sim \frac{\sqrt{m}}{(m+1)!}\int_\R
|\Gamma(\tfrac{1}{2}+i\gamma)|\dd\gamma$ when $m$ tends to infinity, in this way we cannot produce an
upper-bound for $|\psi^{(m)}_\K(x) - \frac{x^{m+1}}{m+1}|$ with a coefficient for $\log\disc_\K$ better
than $x^{m+1/2}(\frac{1}{2\pi\sqrt{m}}+o(1))\int_\R |\Gamma(\tfrac{1}{2}+i\gamma)|\dd\gamma$.\\
Iterating $m$ times the partial summation for the logarithm of the residue of $\zeta_\K$ we get a
remainder term which, in its main part, is controlled by $2(m+1)!\sum_j a_j$, so that it tends to
infinity as $\frac{\sqrt{m}}{\pi}\int_\R |\Gamma(\tfrac{1}{2}+i\gamma)|\dd\gamma$: this proves that one
cannot expect to improve the algorithm for the residue simply increasing $m$. A closer look at the
sequence $(m+1)!\int_\R g_m(\gamma)\dd\gamma$ shows that it attains its minimum exactly when $m=2$, so
that our formulas are already the best we can produce.
\end{remark}

\begin{proof}[Proof of Lemma~\ref{lem:A4}]
We still follow the method described in the introduction. We use $s_j=1+j/2$ as in Lemma~\ref{lem:B3}.
Let $g(\gamma):= 2/(1+4\gamma^2)$, so that $|B_\K| = \sum_\gamma g(\gamma)$. Then using Sturm's algorithm
we see that $g(\gamma) \leq \sum_{j=1}^{10} a_j f(s_j,\gamma) $ for every $\gamma\in\R$, when the
constants $a_j$ have the values in Table~\ref{tab:Bbho3}.
As for Lemma~\ref{lem:B3} the numbers $a_j$ have been generated imposing a double contact at the points
in $\Upsilon:=\{0.84,2.04,4.01,9.61\}$, the equality at $\gamma=0$ and the asymptotic equality for
$\gamma\to\infty$. With these constants we have

\begin{equation}\label{eq:B28bis}
\begin{aligned}
&\sum_j a_j                                                =     0.51543\ldots,&\qquad&\sum_j a_j\Big(\frac{2}{s_j}+\frac{2}{s_j-1}\Big)         \leq  9.3419,\\
&\sum_j a_j\frac{\Gamma'}{\Gamma}\Big(\frac{s_j}{2}\Big)   \leq -1.0094,       &\qquad&\sum_j a_j\frac{\Gamma'}{\Gamma}\Big(\frac{s_j+1}{2}\Big) \leq -0.297 ,\\
&\sum_j a_j\frac{\zeta'_\K}{\zeta_\K}(s_j) \leq  0,
\end{aligned}
\end{equation}
where the last inequality follows by noticing once again that it is $-\sum_n\tilde\Lambda_\K(n) S(n)$
with $S(n):=\sum_j \frac{a_j}{n^{s_j}}$, and that each $S(n)$ is positive (for $n < 150$ by numerical
test, and for every $n\geq 150$ because the sums in pairs $\frac{a_1}{n^{s_1}} + \frac{a_2}{n^{s_2}}$,
\ldots, $\frac{a_{9}}{n^{s_{9}}} + \frac{a_{10}}{n^{s_{10}}}$ are positive). The result now follows
from~\eqref{eq:B17} and~\eqref{eq:B28bis}.
\end{proof}

\begin{remark}\label{rem:B3}
The best coefficient of $\log\disc_\K$ we can get from our argument is $\tfrac{1}{2}$. Moreover, trying
to find a lower bound, we can prove $|B_\K| \geq 0.4512\log\disc_\K - 5.2554 n_\K + 5.2784$.
Unfortunately this bound is not sufficiently strong to produce anything useful for our purposes, thus
we do not include its proof.
\end{remark}

\section{Proof of Corollary~\ref{cor:B1}}\label{sec:B4}
\begin{proof}[Proof of the case $\kappa=0$.]
We write
\[
\psi^{(1)}_\K(x) = \sum_{\substack{\PP,\,m\\ \Norm\PP^m\leq x}} \log(\Norm\PP)(x-\Norm\PP^m)
\]
as $S_1+S_2$, where $S_1$ is the contribution to $\psi^{(1)}_\K(x)$ coming from the primes in the
statement, and $S_2$ is the complementary term. Thus
\[
S_1 := \sum_{\substack{\PP\\ \Norm\PP\text{ prime}}}\log \Norm\PP\sum_{\substack{m\\ \Norm\PP^m\leq x}}\,(x-\Norm\PP^m)
     = \sum_{p\leq x}\Big(\sum_{\substack{\PP|p\\\Norm\PP=p}}1\Big)\log p\sum_{\substack{m\\ p^m\leq x}}\,(x-p^m)
\]
and
\begin{align*}
S_2 := \sum_{\substack{\PP\\ \Norm\PP\text{ not prime}}}\log \Norm\PP\sum_{\substack{m\\ \Norm\PP^m\leq x}}\,(x-\Norm\PP^m)
     = \sum_{p\leq x}\sum_{\substack{\PP|p\\\Norm\PP=p^{f_\PP}, f_\PP\geq 2}}\!\!\!\!\!\!\!\!\!f_\PP\log p\sum_{\substack{m\\ p^{mf_\PP}\leq x}}\,(x-p^{mf_\PP}),
\end{align*}
where $f_\PP$ is the residual degree of the prime ideal $\PP$. The definition of $S_2$ shows that
\begin{align}
S_2
&\leq \sum_{p\leq x}\sum_{\substack{\PP|p\\\Norm\PP=p^{f_\PP}, f_\PP\geq 2}}f_\PP\log p\sum_{\substack{m\\ p^{m}\leq \sqrt{x}}}\,(x-p^{2m})
 \leq n_\K\sum_{p}\sum_{\substack{m\\ p^m\leq \sqrt{x}}}\log p(x-p^{2m})                      \notag \\
&=    n_\K\sum_{n\leq \sqrt{x}}\Lambda(n)(x-n^2)
 =    n_\K(2 \sqrt{x}\psi^{(1)}_\Q(\sqrt{x})
           -\psi^{(2)}_\Q(\sqrt{x}))                                                          \notag \\
&\leq n_\K\Big(\frac{2}{3}x^{3/2}
               + 2\sqrt{x}\Big|\psi^{(1)}_\Q(\sqrt{x})- \frac{x}{2}\Big|
               + \Big|\psi^{(2)}_\Q(\sqrt{x})- \frac{x^{3/2}}{3}\Big|
          \Big).                                                                              \label{eq:B29}
\end{align}
Thus, in order to prove that $S_1$ is positive it is sufficient to verify that $\psi^{(1)}_\K(x)$ is
larger than the function appearing on the right in~\eqref{eq:B29}, which can be estimated using the upper
bounds for $\Q$ and the lower bound for $\psi^{(1)}_\K(x)$ in Theorem~\ref{th:B2}. After some
simplifications the inequality is reduced to
\[
\sqrt{x} \geq \mathcal{L}_\K = 1.075(\log\disc_\K + 13) > A
\]
where
\begin{align}
A :=& 2(0.5375\log\disc_\K  - 1.0355n_\K + 5.3879)
      + 2(n_\K-1)\frac{\log x}{\sqrt{x}}                             \label{eq:ABA}\\
    & + \frac{2}{\sqrt{x}}(1.0155\log\disc_\K -2.1041 n_\K + 8.3419)
      + \frac{2}{x^{3/2}}(\log\disc_\K - 1.415 n_\K + 4)             \notag        \\
    & + 2n_\K\Big(\frac{2}{3}
                  + \frac{0.0939}{x^{1/4}}
                  + \frac{5.514 }{x^{1/2}}
             \Big).                                                  \notag
\end{align}
%
After some rearrangements the inequality $\mathcal{L}_\K > A$ becomes
\begin{multline*}
1.5996 + \frac{\log x}{\sqrt{x}}
\geq \frac{1}{\sqrt{x}}(1.0155\log\disc_\K + 8.3419)
     + \frac{1}{x^{3/2}}(\log\disc_\K + 4)                          \\
     + n_\K\Big(-0.3688
                 + \frac{\log x}{\sqrt{x}}
                 + \frac{0.0939}{x^{1/4}}
                 + \frac{3.4099}{x^{1/2}}
                 - \frac{1.415 }{x^{3/2}}
            \Big)
\end{multline*}
which is implied by the simpler
\begin{equation}\label{eq:Bbho}
0.6546 + \frac{\log x}{\sqrt{x}}
\geq n_\K\Big(-0.3688
                 + \frac{\log x}{\sqrt{x}}
                 + \frac{0.0939}{x^{1/4}}
                 + \frac{3.4099}{x^{1/2}}
                 - \frac{1.415 }{x^{3/2}}
         \Big)
\end{equation}
%
because
\[
\frac{1.0155\log\disc_\K + 8.3419}{\sqrt{x}}
+\frac{\log\disc_\K + 4}{x^{3/2}}
\leq 0.945
\]
under the assumption $\sqrt{x}\geq 1.075(\log\disc_\K+13)$.
%
%
The function appearing on the right-hand side of~\eqref{eq:Bbho} is negative for $\sqrt{x}\geq 30$
%
%
and this is enough to prove the inequality when $\log\disc_\K \geq 15$.
If $\log\disc_\K\leq 15$, Odlyzko's Table~3 \cite{OdlyzkoTables} of inequalities for the discriminant
shows that this may happen only for $n_\K \leq 8$.
%
For every $n_\K \leq 8$ Inequality~\eqref{eq:Bbho} holds when $x\geq \bar{x}$ for a suitable constant
$\bar{x}$ depending on $n_\K$. However, for each $n_\K$ there is a minimal value $\bar{x}_{\min}$ for
$x$, coming from the minimal discriminant for that degree (estimated again using Odlyzko's table). Values
for $\bar{x}$ and $\bar{x}_{\min}$ are shown in Table~\ref{tab:bho}: in every case $\bar{x} <
\bar{x}_{\min}$, thus proving~\eqref{eq:Bbho} also for $n_\K\leq 8$.
%
%
\end{proof}

\begin{proof}[Proof of the general case.]
Let $\mathcal{A}$ be the set of all degree-one prime ideals in $\mathcal{O}_\K$. Thus the term $S_1$
appearing in the decomposition of $\psi^{(1)}_\K(x)$ as $S_1+S_2$ in the proof of the case $\kappa=0$
reads
\[
S_1 = \sum_{\substack{\PP\\ \Norm\PP\leq x}}\delta_{\PP\in \mathcal{A}}\log\Norm\PP\sum_{\substack{m\\ \Norm\PP^m\leq x}}\,(x-\Norm\PP^m)
\]
where $\delta_{\PP\in \mathcal{A}}$ is $1$ if $\PP\in \mathcal{A}$ and $0$ otherwise. With two
applications of the Cauchy--Schwarz inequality we get
\begin{align*}
S_1 &\leq \Big(\sum_{\substack{\PP\\ \Norm\PP\leq x}}\delta_{\PP\in \mathcal{A}}\Big)^{1/2}
          \cdot
          \Big(\sum_{\PP}\log^2\Norm\PP\Big(\sum_{\substack{m\\ \Norm\PP^m\leq x}}\,(x-\Norm\PP^m)\Big)^2\Big)^{1/2}\\
    &\leq \Big(\sum_{\substack{\PP\\ \Norm\PP\leq x}}\delta_{\PP\in \mathcal{A}}\Big)^{1/2}
          \cdot
          \Big(\sum_{\PP}\log^2\Norm\PP\intpart{\frac{\log x}{\log\Norm\PP}}\sum_{\substack{m\\ \Norm\PP^m\leq x}}\,(x-\Norm\PP^m)^2\Big)^{1/2}\\
    &\leq \Big(\sum_{\substack{\PP\\ \Norm\PP\leq x}}\delta_{\PP\in \mathcal{A}}\Big)^{1/2}
          \cdot \sqrt{\log x}
          \Big(\sum_{\PP}\log\Norm\PP\sum_{\substack{m\\ \Norm\PP^m\leq x}}\,(x-\Norm\PP^m)^2\Big)^{1/2}  \\
    &=    \Big(\sum_{\substack{\PP\\ \Norm\PP\leq x}}\delta_{\PP\in \mathcal{A}}\Big)^{1/2}
          \cdot \sqrt{\log x \psi^{(2)}_\K(x)}.
\end{align*}
Thus, in order to have $\sum_{\substack{\Norm\PP\leq x}}\delta_{\PP\in \mathcal{A}} > \kappa$ it is
sufficient to have $S_1 > \sqrt{\kappa \log x\,\psi^{(2)}_\K(x)}$, i.e.
\[
\psi^{(1)}_\K(x) > S_2 + \sqrt{\kappa \log x\ \psi^{(2)}_\K(x)}.
\]
Recalling the upper bound~\eqref{eq:B29} for $S_2$ and Theorem~\ref{th:B2} (with $\K\neq \Q$), for the
previous inequality it is sufficient to have
\[
\sqrt{x} > A + 2\sqrt{\kappa B\log x}
\]
where $A$ is given in~\eqref{eq:ABA} and
\begin{align*}
B :=& \frac{1}{3} + \frac{1}{\sqrt{x}}(0.3526\log\disc_\K - 0.8212 n_\K + 4.4992)
       + (n_\K-1)\frac{1}{x}\Big(\log x - \frac{1}{2}\Big)                        \\
    &  + \frac{1}{x}(1.0155\log \disc_\K - 2.1041 n_\K + 8.3419)
       + \frac{2}{x^2}(\log \disc_\K - 1.415 n_\K + 4)                            \\
    &  + \frac{1}{x^3}(\log \disc_\K - 0.9151 n_\K + 2).
\end{align*}
\noindent%
We can take $\sqrt{x}=\mathcal{L}_\K + \sqrt{8\kappa\log(\mathcal{L}_\K +
\sqrt[3]{\kappa}\log\kappa)}$ with $\mathcal{L}_\K = 1.075(\log\disc_\K+13)$, and under this hypothesis
function $B$ is bounded by $2/3$. To prove it we notice that
\[
\frac{1}{x}\Big(\log x - \frac{1}{2}\Big)
\leq \frac{0.33}{\sqrt{x}}
\]
because $\sqrt{x}\geq \mathcal{L}_\K\geq 15$. This remark and the assumption $n_\K \geq 2$ show that $B$
is smaller than
\begin{align*}
B \leq
    & \frac{1}{3} + \frac{1}{\sqrt{x}}(0.3526\log\disc_\K + 3.6)
      + \frac{1}{x}(1.0155\log \disc_\K + 4.2)
      + \frac{2}{x^2}(\log \disc_\K + 1.2)                                            \\
    & + \frac{1}{x^3}(\log \disc_\K + 0.2).
\end{align*}
It is now easy to prove that this is smaller than $2/3$ for $\sqrt{x}\geq \mathcal{L}_\K$.\\
%
%
Since $B\leq \frac{2}{3}$ we only need to prove that
\begin{align*}
\sqrt{x} > A + 2\sqrt{\frac{2}{3}}\sqrt{\kappa\log x}.
\end{align*}
From the proof of Corollary~\ref{cor:B1} we already know that $\mathcal{L}_\K>A$. Thus the inequality
holds when $\kappa=0$ and for $\kappa>0$ it is sufficient to verify that
\[
(\mathcal{L}_\K + \sqrt[3]{\kappa}\log\kappa)^{3/2}
\geq
\mathcal{L}_\K + (8\kappa \log(\mathcal{L}_\K + \sqrt[3]{\kappa}\log\kappa))^{1/2}
\]
which holds true for every $\mathcal{L}_\K \geq 15$ and every $\kappa>0$.
%
%
\end{proof}

\section{Proof of Corollary~\ref{cor:Bintro} and improvements}\label{sec:B5}
Starting with~\eqref{eq:B3} and with, respectively, one and two further integrations by parts one gets
{
\begin{subequations}\label{eq:C0}
\begin{align}
\log\underset{s=1}{\res}\zeta_\K(s)
&= \sum_{n\leq N} \big(\tilde{\Lambda}_\K(n)-\Lambda_\Q(n)\big) W^{(1)}(n,N) +\RR^{(1)}(N), \label{eq:C0a}\\
\log\underset{s=1}{\res}\zeta_\K(s)
&= \sum_{n\leq N} \big(\tilde{\Lambda}_\K(n)-\Lambda_\Q(n)\big) W^{(2)}(n,N) +\RR^{(2)}(N)  \label{eq:C0b}
\end{align}
\end{subequations}
}
with the weights
\begin{align*}
W^{(1)}(n,N) &:=  f(n) - f(N) - (n-N)f'(N),                          \\
W^{(2)}(n,N) &:= f(n) - f(N) - (n-N)f'(N) - \frac{1}{2}(n-N)^2f''(N)
\end{align*}
and the remainders
\begin{subequations}\label{eq:C3}
\begin{align}
\RR^{(1)}(N) &:=             \int_N^{+\infty} (\psi^{(1)}_\K(x)-\psi^{(1)}_\Q(x)) f'' (x)\dd x, \label{eq:C3a}\\
\RR^{(2)}(N) &:= -\frac{1}{2}\int_N^{+\infty} (\psi^{(2)}_\K(x)-\psi^{(2)}_\Q(x)) f'''(x)\dd x, \label{eq:C3b}
\end{align}
\end{subequations}
giving immediately the bounds
\begin{subequations}\label{eq:C1}
\begin{align}
\big|\RR^{(1)}(N)\big| &\leq \int_N^{+\infty} |\psi^{(1)}_\K(x)-\psi^{(1)}_\Q(x)|\cdot|f''(x)|\dd x,            \label{eq:C1a}\\
\big|\RR^{(2)}(N)\big| &\leq \frac{1}{2}\int_N^{+\infty} |\psi^{(2)}_\K(x)-\psi^{(2)}_\Q(x)|\cdot|f'''(x)|\dd x.\label{eq:C1b}
\end{align}
\end{subequations}
We can now prove
\begin{corollary}\label{cor:B3}
(GRH) In Equations~\eqref{eq:C0a} and \eqref{eq:C0b} the remainders satisfy
\begin{equation}\label{eq:D2}
|\RR^{(1)}(N)| \leq \RR^{(1)}_{\bas}(N)
\quad\text{and}\quad
|\RR^{(2)}(N)| \leq \RR^{(2)}_{\bas}(N)
\qquad
\forall N\geq 3,
\end{equation}
with
\begin{subequations}\label{eq:D4}
\begin{align}
\RR&^{(1)}_{\bas}(N)
:= \alpha^{(1)}_\K  \Big(\frac{\frac{5}{2}+y}{\sqrt{N}\log N} + \frac{3}{4} \ExpInt\Big(\frac{1}{2}\log N\Big)\Big)
      + \beta^{(1)}_\K \frac{2+3y}{N}                             \label{eq:D4a}\\
&     + \gamma^{(1)}_\K\frac{2y+y^2}{N}
      + \delta^{(1)}_\K\frac{y+y^2}{N^2},                         \notag\\
\RR&^{(2)}_{\bas}(N)
:=  \alpha^{(2)}_\K\Big(\frac{\frac{33}{8} + \frac{11}{4}y + y^2}{\sqrt{N}\log N} + \frac{15}{16}\ExpInt\Big(\frac{1}{2}\log N\Big)\Big)
      + \beta^{(2)}_\K   \frac{3+\frac{11}{2}y+\frac{3}{2}y^2}{N} \label{eq:D4b}\\
  &   + \gamma^{(2)}_\K  \frac{3y+\frac{5}{2}y^2+y^3}{N}
      + \delta^{(2)}_\K  \frac{\frac{3}{2}y+2y^2+y^3}{N^2}
      + \eta^{(2)}_\K    \frac{y+\frac{3}{2}y^2+y^3}{N^3}         \notag
\end{align}
\end{subequations}
where $\ExpInt(x):=\int_1^{+\infty}e^{-xt}t^{-1}\dd t$ is the exponential integral, $y:=(\log N)^{-1}$ and
\begin{align*}
   \alpha^{(1)}_\K   &= 0.5375\log\disc_\K - 1.0355 n_\K + 5.4341,
&&&\beta^{(1)}_\K    &= n_\K-1,                                      \\
   \gamma^{(1)}_\K   &= 1.0155\log \disc_\K - 2.1041 n_\K + 10.1799,
&&&\delta^{(1)}_\K   &=       \log\disc_\K - 1.415 n_\K + 4,         \\\\
   \alpha^{(2)}_\K   &= 0.3526\log\disc_\K - 0.8212 n_\K + 4.5007,
&&&\beta^{(2)}_\K    &= n_\K-1,                                      \\
   \gamma^{(2)}_\K   &= 1.0155\log \disc_\K - 2.6041 n_\K + 10.6799,
&&&\delta^{(2)}_\K   &= 2\log \disc_\K - 2.83 n_\K + 8,              \\
\eta^{(2)}_\K        &= \log \disc_\K - 0.9151 n_\K + 2.
\end{align*}
\end{corollary}
\begin{proof}
Suppose we have found constants $\alpha^{(1)}_\K$,\ldots,$\delta^{(1)}_\K$ and
$\alpha^{(2)}_\K$,\ldots,$\eta^{(2)}_\K$ such that
\begin{subequations}\label{eq:C2}
\begin{align}
|\psi^{(1)}_\K(x) - \psi^{(1)}_\Q(x)|
&\leq  \alpha^{(1)}_\K x^{3/2}
     + \beta^{(1)}_\K  x\log x
     + \gamma^{(1)}_\K x
     + \delta^{(1)}_\K,                \label{eq:C2a}\\
|\psi^{(2)}_\K(x) - \psi^{(2)}_\Q(x)|
&\leq  \alpha^{(2)}_\K x^{5/2}
     + \beta^{(2)}_\K  x^2\log x
     + \gamma^{(2)}_\K x^2
     + \delta^{(2)}_\K x
     + \eta^{(2)}_\K.                  \label{eq:C2b}
\end{align}
\end{subequations}
For~\eqref{eq:D4a} we plug \eqref{eq:C2a} into~\eqref{eq:C1a} and we use (\ref{eq:B30b}--\ref{eq:B30j}):
the integrals apply here because $f(x)=(x\log x)^{-1}$ is a completely monotone function, i.e. satisfies
$(-1)^k f^{(k)}(x)$ $>0$ for every $x>1$ and for every order $k$.\\
For~\eqref{eq:D4b} we plug \eqref{eq:C2b} into~\eqref{eq:C1b} and we use (\ref{eq:B30a}--\ref{eq:B30l}).\\
The existence and the values of the constants $\alpha^{(j)}_\K,\ldots$ are an immediate consequence of
Theorem~\ref{th:B2}.
\end{proof}
\noindent Coming back to the remark below Corollary~\ref{cor:Bintro} this strategy produces algorithms
where the errors $|\RR^{(1)}(N)|$ and $|\RR^{(2)}(N)|$ are bounded essentially by $2.15
\frac{\log\disc_\K}{\sqrt{N}\log N}$, and $2.116 \frac{\log\disc_\K}{\sqrt{N}\log N}$, respectively. The
minimal $N$ needed for Buchmann's algorithm using Belabas and Friedman's result and ours are compared in
Table~\ref{tab:B3}.
\medskip

The terms $-xr_\K$ and $R^{(1)}_{r_1,r_2}(x)$ in~\eqref{eq:E1a} and $-x^2r_\K$ and $R^{(2)}_{r_1,r_2}(x)$
in~\eqref{eq:E1b} are generally of comparable size and opposite in sign for the typical values of $x$
which are needed in this application; thus it is possible to improve the result by estimating the
remainders in such a way as to keep these terms together. This remark produces the following corollary.

\begin{corollary}\label{cor:B4}
(GRH) In Equations~\eqref{eq:C0a} and \eqref{eq:C0b} the remainders satisfy
\begin{equation}\label{eq:B8}
|\RR^{(1)}(N)| \leq \RR^{(1)}_{\imp}(N)
\quad\text{and}\quad
|\RR^{(2)}(N)| \leq \RR^{(2)}_{\imp}(N)
\qquad
\forall N\geq 3,
\end{equation}
where
\begin{align}
\RR^{(1)}_{\imp}(N)
&:= \alpha^{(1)}_\K\Big(\frac{\frac{5}{2}+y}{\sqrt{N}\log N} + \frac{3}{4} \ExpInt\Big(\frac{1}{2}\log N\Big)\Big)
     +\Big(d_\K + \frac{r_2}{4N}\Big)\frac{y^2}{N}                                                  \label{eq:B9} \\
  &  +\Big|d_\K           \frac{2 + y - y^2}{N}
           + (r_\K-r_\Q)  \frac{2y + y^2}{N}
           - r_2          \frac{1 + \tfrac{5}{2}y + y^2}{N^2}
           - (r'_\K-r'_\Q)\frac{y + y^2}{N^2}
      \Big|,                                                                                        \notag        \\
\RR^{(2)}_{\imp}(N)
&:=
    \alpha^{(2)}_\K\Big(\frac{\frac{33}{8} + \frac{11}{4}y + y^2}{\sqrt{N}\log N} + \frac{15}{16}\ExpInt\Big(\frac{1}{2}\log N\Big)\Big)
     +\Big(d_\K + \frac{r_2}{4N}\Big)\frac{y^2}{N}\Big(1+\frac{5}{yN^2}\Big)                        \label{eq:B10}\\
  &  +\frac{1}{2}\Big|d_\K             \frac{6+2y-\tfrac{9}{2}y^2-3y^3}{N}
                      + (r_\K-r_\Q)    \frac{6y + 5y^2 + 2y^3}{N}
                      - 2r_2           \frac{3+\tfrac{11}{2}y+3y^2}{N^2}                            \notag        \\
  &                   - 2(r'_\K-r'_\Q) \frac{3y + 4y^2 + 2y^3}{N^2}
                      + d_\K           \frac{2+\tfrac{20}{3}y+\tfrac{15}{2}y^2+3y^3}{N^3}
                      + (r''_\K-r''_\Q)\frac{2y + 3y^2 + 2y^3}{N^3}
                 \Big|,                                                                             \notag
\end{align}
and $\alpha^{(1)}_\K$ and $\alpha^{(2)}_\K$ are as in Corollary~\ref{cor:B3}.
\end{corollary}
\begin{proof}
By~\eqref{eq:C3a} and the explicit formula~\eqref{eq:E1a} we get
\begin{multline*}
|\RR^{(1)}(N)| = \Big|\int_N^{+\infty} \Big(\sum_{\substack{\rho\\\zeta_\Q(\rho)=0}} \frac{x^{\rho+1}}{\rho(\rho+1)}
                                        \ - \sum_{\substack{\rho\\\zeta_\K(\rho)=0}} \frac{x^{\rho+1}}{\rho(\rho+1)}\\
                                          -(r_\K-r_\Q)x +(r'_\K-r'_\Q) + R^{(1)}_{r_1,r_2}(x)-R^{(1)}_{1,0}(x)
                                       \Big)f''(x)\dd x \Big|.
\end{multline*}
Here we isolate the part depending on the zeros. We estimate it by moving the absolute value in the inner part
both of the integral and of the sum, and then applying the upper bound in Lemma~\ref{lem:B3}. In this way
we get
\begin{multline*}
|\RR^{(1)}(N)|
\leq \alpha^{(1)}_\K\int_N^{+\infty} x^{3/2}|f''(x)|\dd x\\
    +\Big|\int_N^{+\infty} \big(-(r_\K-r_\Q)x +(r'_\K-r'_\Q) + R^{(1)}_{r_1,r_2}(x)-R^{(1)}_{1,0}(x)\big)f''(x)\dd x \Big|
\end{multline*}
where $\alpha^{(1)}_\K$ is the constant of Corollary~\ref{cor:B3}. We apply then
Equalities~(\ref{eq:B30b}--\ref{eq:B30f}), thus getting
\begin{multline*}
|\RR^{(1)}(N)|
\leq \frac{\alpha^{(1)}_\K}{\sqrt{N}}(4y-2y^2+12y^3)\\
    +\Big|-\frac{r_\K-r_\Q}{N}(2y + y^2) + \frac{r'_\K-r'_\Q}{N^2}(y + y^2)
     +\int_N^{+\infty} \big(R^{(1)}_{r_1,r_2}(x)-R^{(1)}_{1,0}(x)\big)f''(x)\dd x \Big|.
\end{multline*}
Recalling the definition of functions $f^{(1)}_j(x)$ and $R^{(1)}_{r_1,r_2}(x)$ in Lemma~\ref{lem:B2} we
have
\begin{multline*}
|\RR^{(1)}(N)|
\leq  \frac{\alpha^{(1)}_\K}{\sqrt{N}}(4y-2y^2+12y^3)
     +\Big|\int_N^{+\infty} (d_\K f^{(1)}_1(x)+r_2f^{(1)}_2(x))f''(x)\dd x \Big|\\
     +\Big|\frac{r_\K-r_\Q}{N}(2y + y^2)
           - \frac{r'_\K-r'_\Q}{N^2}(y + y^2)
           + \int_N^{+\infty} \big(d_\K x(\log x -1) - r_2(\log x +1)\big)f''(x)\dd x
      \Big|.
\end{multline*}
The part depending on $f^{(1)}_j$ functions is estimated using the inequalities $0 < f^{(1)}_1(x)\leq
0.6 x^{-1}$ and $0<f^{(1)}_2(x)\leq 0.2 x^{-2}$ for $x\geq 3$, the other integrals are computed
via~(\ref{eq:B30d}--\ref{eq:B30l}). After some computations one gets the bound
$|\RR^{(1)}(N)|\leq \RR^{(1)}_{\imp}(N)$ with $\RR^{(1)}_{\imp}(N)$ given in~\eqref{eq:B9}.\\
The proof of~\eqref{eq:B10} is similar using $0 < f^{(2)}_1(x)\leq 0.1 x^{-2}$ and $0<f^{(2)}_2(x)\leq
0.4 x^{-1}$ for $x\geq 3$, and~(\ref{eq:B30c}--\ref{eq:B30i}).
\end{proof}
In order to apply the formulas in Corollary~\ref{cor:B4} we recall that
\[
r_\Q   = \log 2\pi
\qquad
r'_\Q  = -\frac{\zeta'}{\zeta}(2)
         +\gamma+\log 2\pi-1
\qquad
r''_\Q = -\frac{\zeta'}{\zeta}(3)
         +\gamma+\log 2\pi-\frac{3}{2}
\]
(for $r_\Q$ see~\cite[Ch.~12]{Davenport1}, the other two are immediate consequence
of~(\ref{eq:B19}--\ref{eq:B20})) but we need also the parameters $r_\K$, $r'_\K$ and $r''_\K$. They can
be estimated as (see the proof of Lemma~\ref{lem:B1})
\begin{subequations}\label{eq:F1}
\begin{align}
-1.0155\log \disc_\K +2.1042 n_\K - 8.3419
&\leq
r_\K
\leq -\tfrac{1}{2}\log \disc_\K + 1.2076 n_\K +1 \label{eq:F1a}\\
-\log \disc_\K + 1.415 n_\K
&\leq
r'_\K
\leq -\log \disc_\K + 1.9851 n_\K                \label{eq:F1b}\\
-\log\disc_\K + 0.9151n_\K
&\leq
r''_\K
\leq -\log\disc_\K + 1.08 n_\K.                  \label{eq:F1c}
\end{align}
\end{subequations}
Thus we can take the largest value that $\RR^{(m)}_{\imp}$ assumes when the parameters run in those
ranges. To that effect, it is sufficient to consider the values of the term in the absolute value where
$r_{\K}$, $r'_{\K}$ and $r''_{\K}$ are replaced by the maximum and the minimum of their range. The
results are summarized in Tables~\ref{tab:B1}--\ref{tab:B4}. Tables~\ref{tab:B1} and~\ref{tab:B2} show
that in any case the improved estimate beats the plain bound by a quantity which largely depends on the
quotient $n_\K/\log\disc_\K$, reaching a gain greater than $10\%$ for $\RR^{(1)}$ and $16\%$ for
$\RR^{(2)}$ for some combinations. This behavior agrees with our motivations for the improved formulas:
keeping together the quantities $d_\K + r_\K y$, $r_2 + r'_\K y$ (for non-totally real fields) and $d_\K
+ r''_\K y$, which are $\approx n_\K - \frac{\log\disc_\K}{\log N}$ (times suitable multiple of
$N^{-1}$), we take advantage of their cancellations which can be quite large for suitable values of
$n_\K/\log\disc_\K$. Tables~\ref{tab:B3} and~\ref{tab:B4} show that the new algorithms improve
Belabas--Friedman's bound by a factor which is at least 3 and sometimes 10. Finally,
Tables~\ref{tab:B3}--\ref{tab:B4} show that in that range of discriminants and for degrees larger than
$10$ it is convenient to use $\RR^{(2)}_{\imp}$ instead of $\RR^{(1)}_{\imp}$.
\smallskip

We could improve the algorithm a bit further by using the relation
\begin{equation}\label{eq:A19}
r_\K = \sum_{n=1}^{+\infty}\frac{\tilde{\Lambda}_\K(n) - \Lambda(n)}{n}\,
       - \log\disc_\K + (\gamma + \log 2\pi)n_\K - \gamma,
\end{equation}
which follows combining the functional equations for $\zeta_\K$ and $\zeta_\Q$. In fact, truncating the
series at a new level $N'$ and estimating the remainder as in~\eqref{eq:C0} via Theorem~\ref{th:B2} we
get an explicit formula which already for $N'\approx 100$ gives for $r_\K$ a range shorter
than~\eqref{eq:F1a}. This computation takes only a small fraction of the total time needed for Buchmann's
algorithm, and the new range allows us to improve the $N$ computed via $\RR^{(m)}_{\imp}$ by a quantity
which in our tests has been generally around $1$--$2\%$, and occasionally large as $5\%$.\\
We can also compute $r'_\K$ and $r''_\K$ via~\eqref{eq:B19} and~\eqref{eq:B20}, but their
ranges~\eqref{eq:F1b} and~\eqref{eq:F1c} are already tight and in the formulas for $\RR^{(m)}_{\imp}$
these parameters appear only in terms which are several orders lower than the principal one, and no
improvement comes from their computation.
\vfill
\appendix
\section{Some integrals}
We collect here a lot of computations and approximations of integrals that are used in
Section~\ref{sec:B5}; they can easily be proved by integration by parts. Recall that $f(x)=(x\log
x)^{-1}$, $N\geq 3$ and $y=(\log N)^{-1}$. Thus
\begin{align*}
f(N) = \frac{y}{N}
\qquad
f'(N) = - \frac{y+y^2}{N^2}
\qquad
f''(N) = \frac{2y+3y^2+2y^3}{N^3}.
\end{align*}
In the following $\theta$ is a constant in $(0,1)$, with possibly different values in each occurrence. We
have
\allowdisplaybreaks[4]
\begin{subequations}\label{eq:B30}
\begin{align}
\int_N^{+\infty} x^{3/2}f''(x)\dd x
&= \frac{1}{\sqrt{N}}(\tfrac{5}{2}y + y^2)                       + \tfrac{3}{4} \ExpInt\big(\tfrac{1}{2}\log N\big) \label{eq:B30b}\\
\int_N^{+\infty} xf''(x)\dd x
& = \frac{1}{N}(2y+y^2)                                             \label{eq:B30d}\\
\int_N^{+\infty} f''(x)\dd x
& = \frac{1}{N^2}(y + y^2)                                          \label{eq:B30f}\\
\int_N^{+\infty} x\log x f''(x)\dd x
& =  \frac{1}{N}(2 +3y - \theta y^2)                                \label{eq:B30j}\\
\int_N^{+\infty} \log x f''(x)\dd x
& =  \frac{1}{N^2}(1+\tfrac{3}{2}y + \tfrac{\theta}{4}y^2)          \label{eq:B30h}\\
\int_N^{+\infty} x^{5/2}f'''(x)\dd x
&= -\frac{1}{\sqrt{N}}(\tfrac{33}{4}y + \tfrac{11}{2}y^2 + 2y^3) - \tfrac{15}{8}\ExpInt\big(\tfrac{1}{2}\log N\big) \label{eq:B30a}\\
\int_N^{+\infty} x^{2}f'''(x)\dd x
& = -\frac{1}{N}(6y + 5y^2 + 2y^3)                                  \label{eq:B30c}\\
\int_N^{+\infty} xf'''(x)\dd x
& = -\frac{1}{N^2}(3y + 4y^2 + 2y^3)                                \label{eq:B30e}\\
\int_N^{+\infty} f'''(x)\dd x
& = -\frac{1}{N^3}(2y + 3y^2 + 2y^3)                                \label{eq:B30g}\\
\int_N^{+\infty} x^2\log x f'''(x)\dd x
& =  \frac{1}{N}(-6-11y-3y^2 + 2\theta y^2)                         \label{eq:B30l}\\
\int_N^{+\infty} x\log x f'''(x)\dd x
& =  \frac{1}{N^2}(-3-\tfrac{11}{2}y-3y^2 - \tfrac{\theta}{4}y^2)   \label{eq:B30k}\\
\int_N^{+\infty} \log x f'''(x)\dd x
& =  \frac{1}{N^3}(-2-\tfrac{11}{3}y-3y^2 + \tfrac{2\theta}{9}y^2). \label{eq:B30i}
\end{align}
\end{subequations}

\begin{table}[hb]
\caption{Parameters for~\eqref{eq:Bbho}.}\label{tab:bho}
\begin{tabular}{|l|rrrrrrr|}
  \toprule
  $n_\K$                          & $2$   & $3$   & $4$   & $5$   & $6$   & $7$   & $8$  \\
  \midrule
  $\bar{x}_{\hphantom{\min}}\leq$ & $ 97$ & $179$ & $253$ & $316$ & $369$ & $414$ & $452$\\
  $\bar{x}_{\min}\geq$            & $229$ & $287$ & $363$ & $456$ & $566$ & $694$ & $840$\\
  \bottomrule
\end{tabular}
\end{table}

\begin{table}[H]
\caption{Least $N$ for Buchmann's algorithm: $\RR^{(1)}_{\bas}$ against $\RR^{(1)}_{\imp}$.}
\label{tab:B1}
{\small
\begin{tabular}{|l|rr|rr|rr|rr|rr|}
  \toprule
                        &\mltc{2}{c|}{$n=2$}              &\mltc{2}{c|}{$n=6$}              &\mltc{2}{c|}{$n=10$}             &\mltc{2}{c|}{$n=20$}           &\mltc{2}{c|}{$n=50$}            \\
  \mltc{1}{|c|}{$\disc$}&$\RR^{(1)}_{\bas}$  &$\RR^{(1)}_{\imp}$  &$\RR^{(1)}_{\bas}$  &$\RR^{(1)}_{\imp}$  &$\RR^{(1)}_{\bas}$  &$\RR^{(1)}_{\imp}$  &$\RR^{(1)}_{\bas}$ &$\RR^{(1)}_{\imp}$ &$\RR^{(1)}_{\bas}$ &$\RR^{(1)}_{\imp}$ \\
  \midrule
  $10^5    $ &$   371$&$  361$  &$   211$&$  190$  &\mltc{1}{c}{---}&\mltc{1}{c|}{---}  &\mltc{1}{c}{---}&\mltc{1}{c|}{---}  &\mltc{1}{c}{---}&\mltc{1}{c|}{---}\\
  $10^{10} $ &$   763$&$  752$  &$   529$&$  485$  &$   341$        &$  310$            &\mltc{1}{c}{---}&\mltc{1}{c|}{---}  &\mltc{1}{c}{---}&\mltc{1}{c|}{---}\\
  $10^{20} $ &$  1835$&$ 1824$  &$  1478$&$ 1406$  &$  1159$        &$ 1085$            &\mltc{1}{c}{---}&\mltc{1}{c|}{---}  &\mltc{1}{c}{---}&\mltc{1}{c|}{---}\\
  $10^{50} $ &$  6961$&$ 6950$  &$  6305$&$ 6231$  &$  5678$        &$ 5541$            &$  4248$        &$ 4088$            &\mltc{1}{c}{---}&\mltc{1}{c|}{---}\\
  $10^{100}$ &$ 20776$&$20765$  &$ 19709$&$19634$  &$ 18668$        &$18529$            &$ 16177$        &$15879$            &$  9704$        &$ 9446$          \\
  $10^{200}$ &$ 64950$&$64939$  &$ 63189$&$63114$  &$ 61451$        &$61310$            &$ 57198$        &$56897$            &$ 45269$        &$44710$          \\
  \bottomrule
\end{tabular}
}
\end{table}
\begin{table}[H]
\caption{Least $N$ for Buchmann's algorithm: $\RR^{(2)}_{\bas}$ against $\RR^{(2)}_{\imp}$.}
\label{tab:B2}
{\small
\begin{tabular}{|l|rr|rr|rr|rr|rr|}
  \toprule
                        &\mltc{2}{c|}{$n=2$}              &\mltc{2}{c|}{$n=6$}              &\mltc{2}{c|}{$n=10$}           &\mltc{2}{c|}{$n=20$}           &\mltc{2}{c|}{$n=50$}              \\
  \mltc{1}{|c|}{$\disc$}&$\RR^{(2)}_{\bas}$  &$\RR^{(2)}_{\imp}$  &$\RR^{(2)}_{\bas}$  &$\RR^{(2)}_{\imp}$  &$\RR^{(2)}_{\bas}$  &$\RR^{(2)}_{\imp}$  &$\RR^{(2)}_{\bas}$ &$\RR^{(2)}_{\imp}$ &$\RR^{(2)}_{\bas}$ &$\RR^{(2)}_{\imp}$ \\
  \midrule
  $10^5    $ &$  466$&$  451$  &$   256$&$   221$  &\mltc{1}{c}{---}&\mltc{1}{c|}{---}  &\mltc{1}{c}{---}&\mltc{1}{c|}{---}  &\mltc{1}{c}{---}&\mltc{1}{c|}{---}\\
  $10^{10} $ &$  899$&$  884$  &$   601$&$   531$  &$   369$        &$   317$           &\mltc{1}{c}{---}&\mltc{1}{c|}{---}  &\mltc{1}{c}{---}&\mltc{1}{c|}{---}\\
  $10^{20} $ &$ 2054$&$ 2039$  &$  1607$&$  1504$  &$  1216$        &$  1097$           &\mltc{1}{c}{---}&\mltc{1}{c|}{---}  &\mltc{1}{c}{---}&\mltc{1}{c|}{---}\\
  $10^{50} $ &$ 7444$&$ 7429$  &$  6631$&$  6524$  &$  5862$        &$  5665$           &$  4141$        &$  3886$           &\mltc{1}{c}{---}&\mltc{1}{c|}{---}\\
  $10^{100}$ &$21750$&$21735$  &$ 20435$&$ 20327$  &$ 19158$        &$ 18957$           &$ 16132$        &$ 15700$           &$ 8544$         &$ 8124$          \\
  $10^{200}$ &$67067$&$67051$  &$ 64905$&$ 64795$  &$ 62775$        &$ 62572$           &$ 57592$        &$ 57153$           &$43265$         &$42382$          \\
  \bottomrule
\end{tabular}
}
\end{table}
\begin{table}[H]
\begin{adjustwidth}{-2.75cm}{-2.75cm}
\caption{Least $N$ for Buchmann's algorithm: according to Belabas--Friedman and the new
algorithms with $\RR^{(1)}_{\bas}$ and $\RR^{(2)}_{\bas}$. Belabas--Friedman's data
is reprinted from~\cite{BelabasFriedman}.}
\label{tab:B3}
{\small
\begin{tabular}{|l|rrr|rrr|rrr|rrr|rrr|}
  \toprule
                        &\mltc{3}{c|}{$n=2$}        &\mltc{3}{c|}{$n=6$}        &\mltc{3}{c|}{$n=10$}                                 &\mltc{3}{c|}{$n=20$}                                 &\mltc{3}{c|}{$n=50$}                              \\
  \mltc{1}{|c|}{$\disc$}&B.--F.  &$\RR^{(1)}_{\bas}$  &$\RR^{(2)}_{\bas}$  &B.--F.  &$\RR^{(1)}_{\bas}$  &$\RR^{(2)}_{\bas}$  &B.--F.          &$\RR^{(1)}_{\bas}$          &$\RR^{(2)}_{\bas}$            &B.--F.          &$\RR^{(1)}_{\bas}$          &$\RR^{(2)}_{\bas}$            &B.--F.          &$\RR^{(1)}_{\bas}$          &$\RR^{(2)}_{\bas}$         \\
  \midrule
  $10^5    $ &$  1619$&$   371$&$  466$  &$  1632$&$   211$&$  256$  &\mltc{1}{c}{---}&\mltc{1}{c}{---}&\mltc{1}{c|}{---}  &\mltc{1}{c}{---}&\mltc{1}{c}{---}&\mltc{1}{c|}{---}  &\mltc{1}{c}{---}&\mltc{1}{c}{---}&\mltc{1}{c|}{---}\\
  $10^{10} $ &$  3169$&$   763$&$  899$  &$  3181$&$   529$&$  601$  &$  3194$        &$   341$        &$  369$            &\mltc{1}{c}{---}&\mltc{1}{c}{---}&\mltc{1}{c|}{---}  &\mltc{1}{c}{---}&\mltc{1}{c}{---}&\mltc{1}{c|}{---}\\
  $10^{20} $ &$  6838$&$  1835$&$ 2054$  &$  6850$&$  1478$&$ 1607$  &$  6861$        &$  1159$        &$ 1216$            &\mltc{1}{c}{---}&\mltc{1}{c}{---}&\mltc{1}{c|}{---}  &\mltc{1}{c}{---}&\mltc{1}{c}{---}&\mltc{1}{c|}{---}\\
  $10^{50} $ &$ 21619$&$  6961$&$ 7444$  &$ 21629$&$  6305$&$ 6631$  &$ 21639$        &$  5678$        &$ 5862$            &$ 21665$        &$  4248$        &$ 4141$            &\mltc{1}{c}{---}&\mltc{1}{c}{---}&\mltc{1}{c|}{---}\\
  $10^{100}$ &$ 56332$&$ 20776$&$21750$  &$ 56341$&$ 19709$&$20435$  &$ 56351$        &$ 18668$        &$19158$            &$ 56374$        &$ 16177$        &$16132$            &$ 56445$        &$  9704$        &$ 8544$          \\
  $10^{200}$ &$156151$&$ 64950$&$67067$  &$156160$&$ 63189$&$64905$  &$156169$        &$ 61451$        &$62775$            &$156191$        &$ 57198$        &$57592$            &$156256$        &$ 45269$        &$43265$          \\
  \bottomrule
\end{tabular}
}
\end{adjustwidth}
\end{table}
\enlargethispage{2cm}
\begin{table}[H]
\begin{adjustwidth}{-2.75cm}{-2.75cm}
\caption{Least $N$ for Buchmann's algorithm: according to Belabas--Friedman and the new
algorithms with $\RR^{(1)}_{\imp}$ and $\RR^{(2)}_{\imp}$. Belabas--Friedman's data
is reprinted from~\cite{BelabasFriedman}.}\label{tab:B4}
{\small
\begin{tabular}{|l|rrr|rrr|rrr|rrr|rrr|}
  \toprule
                        &\mltc{3}{c|}{$n=2$}       &\mltc{3}{c|}{$n=6$}         &\mltc{3}{c|}{$n=10$}                                 &\mltc{3}{c|}{$n=20$}                                 &\mltc{3}{c|}{$n=50$}                               \\
  \mltc{1}{|c|}{$\disc$}&B.--F.  &$\RR^{(1)}_{\imp}$   &$\RR^{(2)}_{\imp}$   &B.--F. &$\RR^{(1)}_{\imp}$   &$\RR^{(2)}_{\imp}$   &B.--F. &$\RR^{(1)}_{\imp}$   &$\RR^{(2)}_{\imp}$   &B.--F. &$\RR^{(1)}_{\imp}$   &$\RR^{(2)}_{\imp}$   &B.--F.     &$\RR^{(1)}_{\imp}$      &$\RR^{(2)}_{\imp}$   \\
  \midrule
  $10^5    $ &$  1619$&$  361$&$  451$  &$  1632$&$   190$&$   221$  &\mltc{1}{c}{---}&\mltc{1}{c}{---}&\mltc{1}{c|}{---}  &\mltc{1}{c}{---}&\mltc{1}{c}{---}&\mltc{1}{c|}{---}  &\mltc{1}{c}{---}&\mltc{1}{c}{---}&\mltc{1}{c|}{---}\\
  $10^{10} $ &$  3169$&$  752$&$  884$  &$  3181$&$   485$&$   531$  &$  3194$        &$   310$        &$   317$           &\mltc{1}{c}{---}&\mltc{1}{c}{---}&\mltc{1}{c|}{---}  &\mltc{1}{c}{---}&\mltc{1}{c}{---}&\mltc{1}{c|}{---}\\
  $10^{20} $ &$  6838$&$ 1824$&$ 2039$  &$  6850$&$  1406$&$  1504$  &$  6861$        &$  1085$        &$  1097$           &\mltc{1}{c}{---}&\mltc{1}{c}{---}&\mltc{1}{c|}{---}  &\mltc{1}{c}{---}&\mltc{1}{c}{---}&\mltc{1}{c|}{---}\\
  $10^{50} $ &$ 21619$&$ 6950$&$ 7429$  &$ 21629$&$  6231$&$  6524$  &$ 21639$        &$  5541$        &$  5665$           &$ 21665$        &$  4088$        &$  3886$           &\mltc{1}{c}{---}&\mltc{1}{c}{---}&\mltc{1}{c|}{---}\\
  $10^{100}$ &$ 56332$&$20765$&$21735$  &$ 56341$&$ 19634$&$ 20327$  &$ 56351$        &$ 18529$        &$ 18957$           &$ 56374$        &$ 15879$        &$ 15700$           &$ 56445$        &$  9446$        &$  8124$         \\
  $10^{200}$ &$156151$&$64939$&$67051$  &$156160$&$ 63114$&$ 64795$  &$156169$        &$ 61310$        &$ 62572$           &$156191$        &$ 56897$        &$ 57153$           &$156256$        &$ 44710$        &$ 42382$         \\
  \bottomrule
\end{tabular}
}
\end{adjustwidth}
\end{table}
\newpage
\begin{table}[H]
\begin{adjustwidth}{-2cm}{-2cm}
\caption{Constants for $\sum_\rho |\rho(\rho+1)|^{-1}$ in Lemma~\ref{lem:B3}.}\label{tab:Bbho1}
\smallskip
{\small
\begin{tabular}{|r|r||r|r|}
  \toprule
  $j$     & \mltc{1}{c||}{$a_j\cdot 10^7$}                    &  $j$     &  \mltc{1}{c|}{$a_j\cdot 10^7$}                    \\
  \midrule
  $ 1$    & $                                       250548071$&  $41$    & $  3648003867198618158032688666281279907332926401$\\
  $ 2$    & $                                    -40769390315$&  $42$    & $ -5353733754976758827081327735207850440805276490$\\
  $ 3$    & $                                   5795175723671$&  $43$    & $  7411481592406340412123547436619012373828347148$\\
  $ 4$    & $                                -642251894123528$&  $44$    & $ -9680502044407712819603502062322026576945648999$\\
  $ 5$    & $                               54218815728127329$&  $45$    & $ 11931531864054985817793303920405879163945577143$\\
  $ 6$    & $                            -3508878919641771688$&  $46$    & $-13877909647596266697860364708436232764310634475$\\
  $ 7$    & $                           177001043449933176447$&  $47$    & $ 15232402120827550086363255671423471322396554451$\\
  $ 8$    & $                         -7094015596077453633868$&  $48$    & $-15775334247682258723942059247603410917983570659$\\
  $ 9$    & $                        230165538494597837675083$&  $49$    & $ 15412228661977544619641915478149603219261905955$\\
  $10$    & $                      -6150059294311314135993327$&  $50$    & $-14200388071264097711591911264486344481572166054$\\
  $11$    & $                     137429722146979678372903545$&  $51$    & $ 12334252439899208072837355489763427886826472535$\\
  $12$    & $                   -2603418437013270575777900517$&  $52$    & $-10094621415908831481370162399779502133799211521$\\
  $13$    & $                   42312055609004270243526202076$&  $53$    & $  7779879804978723458319595088819777701055258725$\\
  $14$    & $                 -596228700498573506460507915379$&  $54$    & $ -5642269216814651472704347110867628137752200021$\\
  $15$    & $                 7352229299660977983271586796428$&  $55$    & $  3847449822637486914738835007382155275278296082$\\
  $16$    & $               -79991031610893192700264201347849$&  $56$    & $ -2464415859390293757851604168538551569779024142$\\
  $17$    & $               773449451301413812623754497322110$&  $57$    & $  1481149204040503957548445332963392835676301519$\\
  $18$    & $             -6689469356634480595952166290773419$&  $58$    & $  -834221598218683553855012914220968482480130787$\\
  $19$    & $             52049469989158830787111938359894141$&  $59$    & $   439683909946941169248931270316639282116138102$\\
  $20$    & $           -366215235328303748457085063911062452$&  $60$    & $  -216506399095273447319941898397799604643721683$\\
  $21$    & $           2340727373433875029268033585654013101$&  $61$    & $    99419464389596242263022332671287321846845659$\\
  $22$    & $         -13647569726889979888635481117558851444$&  $62$    & $   -42484842271343000074946235138759717939948915$\\
  $23$    & $          72856233722227845138595374556506130213$&  $63$    & $    16854882803935701426471290624390658493962917$\\
  $24$    & $        -357308755193444577424236430238048816629$&  $64$    & $    -6191128565930702299565499949693343368179853$\\
  $25$    & $        1614746152052189321203222537039119640756$&  $65$    & $     2099040048795249597742846189024188906401966$\\
  $26$    & $       -6742815290893858601169185495146599758906$&  $66$    & $     -654534994786055122946588844807706357840291$\\
  $27$    & $       26081651135346560764877059272421175463793$&  $67$    & $      186947926780001735965997901059271943644593$\\
  $28$    & $      -93662343928951334238283477190373026235970$&  $68$    & $      -48675054083574200340006259345314988135384$\\
  $29$    & $      312909679670641654646206585298379548969664$&  $69$    & $       11488195908804597573088392774662830983598$\\
  $30$    & $     -974320711195668140488233654168711408974431$&  $70$    & $       -2441545378407438626531756675121759469076$\\
  $31$    & $     2832324810202406292456790051806622242980143$&  $71$    & $         463522993226953954733282029125741713819$\\
  $32$    & $    -7698430960693278611182246394416801692267482$&  $72$    & $         -77845294874645427333933115933295893005$\\
  $33$    & $    19591848109684436395280873748247452691475281$&  $73$    & $          11425492896861966116655614216587059464$\\
  $34$    & $   -46741161307608105954759866712283645186774375$&  $74$    & $          -1443037809175186486864654360574474088$\\
  $35$    & $   104654143256889695138455737518470291722254806$&  $75$    & $            153673972446397363248771006965862929$\\
  $36$    & $  -220128737522779177621064610160993365216868168$&  $76$    & $            -13418974865215897151246028213990153$\\
  $37$    & $   435354172671749489946963445292948033362127211$&  $77$    & $               922600572073108333758203469960875$\\
  $38$    & $  -810197596515155479177768566395714272810920262$&  $78$    & $               -46833786494978206937017577663173$\\
  $39$    & $  1419759138597775056528221649897613967888797952$&  $79$    & $                 1560648037479364896275100707017$\\
  $40$    & $ -2344042914614942938251851695053817440107394201$&  $80$    & $                  -25610063982827093894391815027$\\
  \bottomrule
\end{tabular}
}
\end{adjustwidth}
\end{table}
\enlargethispage{5cm}
\begin{table}[H]
\caption{Constants for $\sum_\rho |\rho(\rho+1)(\rho+2)|^{-1}$ in Lemma~\ref{lem:B3}.}\label{tab:Bbho2}
\smallskip
{\small
\begin{tabular}{|r|r||r|r|}
  \toprule
  $j$     & \mltc{1}{c||}{$a_j\cdot 10^7$}    &  $j$    & \mltc{1}{c|}{$a_j\cdot 10^7$}    \\
  \midrule
  $ 1$    & $                      116043280$&  $21$    & $  44212581087391037851257051242$\\
  $ 2$    & $                   -15019134746$&  $22$    & $ -82776719893697522350544625956$\\
  $ 3$    & $                  1306482026256$&  $23$    & $ 136740298375301487499890195367$\\
  $ 4$    & $                -76315741770330$&  $24$    & $-199275732886794715825307355765$\\
  $ 5$    & $               3116274365157230$&  $25$    & $ 255978158207512987528401996503$\\
  $ 6$    & $             -92621169453588672$&  $26$    & $-289344568336337774362395707820$\\
  $ 7$    & $            2074954505670798718$&  $27$    & $ 287063511581435319315875881542$\\
  $ 8$    & $          -36069656819440696263$&  $28$    & $-249076192247094252611008694964$\\
  $ 9$    & $          498302313581120124204$&  $29$    & $ 188100860940650555126546470265$\\
  $10$    & $        -5579712481960141840354$&  $30$    & $-122861964251233620242612405716$\\
  $11$    & $        51471550420429886034202$&  $31$    & $  68841615651370858094138826161$\\
  $12$    & $      -396486283111534949768375$&  $32$    & $ -32737240356857723641726749028$\\
  $13$    & $      2579203445202845079404723$&  $33$    & $  13026982181479475895165888915$\\
  $14$    & $    -14302917461736234191777842$&  $34$    & $  -4255456128181013051676843112$\\
  $15$    & $     68148701393954628073176631$&  $35$    & $   1111002720718102002015316745$\\
  $16$    & $   -280819396505042268256263139$&  $36$    & $   -222834382207437523098078851$\\
  $17$    & $   1006203468485334827305158167$&  $37$    & $     32228595062589755026085278$\\
  $18$    & $  -3148890161469033145131905085$&  $38$    & $     -2991080884530620994922737$\\
  $19$    & $   8637410243724442351566255216$&  $39$    & $       133739429590971377317925$\\
  $20$    & $ -20823652528449395097665316823$&  ---     & \mltc{1}{c|}{---}                \\
  \bottomrule
\end{tabular}
}
\end{table}

\begin{table}[H]
\caption{Constants for $-B_\K=\sum_\rho\rho^{-1}$ in Lemma~\ref{lem:A4}.}\label{tab:Bbho3}
\smallskip
{\small
\begin{tabular}{|r|r||r|r|}
  \toprule
  $j$     & \mltc{1}{c||}{$a_j\cdot 10^7$} &  $j$    & \mltc{1}{c|}{$a_j\cdot 10^7$}    \\
  \midrule
  $ 1$    & $                    149178011$&  $ 6$   & $ -189514259129$                 \\
  $ 2$    & $                  -1773766184$&  $ 7$   & $  205612934195$                 \\
  $ 3$    & $                  11465438478$&  $ 8$   & $ -140312989024$                 \\
  $ 4$    & $                 -45115091060$&  $ 9$   & $   54661946795$                 \\
  $ 5$    & $                 114102793523$&  $10$   & $   -9271031235$                 \\
  \bottomrule
\end{tabular}
}
\end{table}


\end{document}